\documentclass[12pt]{amsart}
\usepackage{amsmath}
\usepackage{amsthm}
\usepackage{amssymb}
\usepackage{mathrsfs}
\usepackage{cite}
\usepackage{amscd,verbatim}
\usepackage[all]{xy}
\usepackage{bm}
\usepackage{url}

\theoremstyle{definition}
\newtheorem{thm}{Theorem}[section]
\newtheorem{definition}[thm]{Definition}
\newtheorem{prop}[thm]{Proposition}
\newtheorem{lem}[thm]{Lemma}
\newtheorem{rem}[thm]{Remark}
\newtheorem{cor}[thm]{Corollary}
\newtheorem{ex}[thm]{Example}

\newtheorem{ques}[thm]{Question}
\newtheorem*{ack}{Acknowledgement}
\newtheorem*{DA}{Data availability}
\newtheorem*{COI}{Conflict of interest}
\numberwithin{equation}{section}

\usepackage[top=20truemm, bottom=20truemm, left=20truemm, right=20truemm]{geometry}

\newcommand{\Z}{\mathbb{Z}}
\newcommand{\Q}{\mathbb{Q}}

\newcommand{\Spec}{\operatorname{Spec}}

\newcommand{\Frac}{\operatorname{Frac}}
\newcommand{\Aut}{{\rm Aut}}
\newcommand{\uAut}{\underline{{\rm Aut}}}

\newcommand{\Dim}{\operatorname{dim}}

\newcommand{\End}{{\rm End}}
\newcommand{\GL}{\mathrm{GL}}
\newcommand{\SL}{\mathrm{SL}}
\newcommand{\fm}{\mathfrak{m}}
\newcommand{\Ker}{\operatorname{Ker}}
\newcommand{\im}{\operatorname{Im}}

\newcommand{\Det}{\operatorname{det}}

\newcommand{\scrO}{\mathscr{O}}

\newcommand{\calG}{\mathcal{G}}
\newcommand{\cB}{\mathcal{B}}

\newcommand{\Def}{\overset{{\rm def}}{=}}

\newcommand{\fppf}{{\rm fppf}}

\newcommand{\frkm}{\mathfrak{m}}
\newcommand{\frkp}{\mathfrak{p}}

\newcommand{\frkP}{\mathfrak{P}}
\newcommand{\frkM}{\mathfrak{M}}

\newcommand{\G}{\mathbb{G}}
\newcommand{\F}{\mathbb{F}}
\newcommand{\A}{\mathbb{A}}

\newcommand{\bfD}{\mathbf{D}}

\newcommand{\Der}{\operatorname{Der}}

\newcommand{\lto}{\longrightarrow}
\newcommand{\Lie}{\operatorname{Lie}}

\begin{document}

\title[Rationality for group schemes]{On retract rationality for \\
finite connected group schemes}
\author[
S. Otabe
]{
Shusuke Otabe
}
\address{Graduate School of Engineering, Nagoya Institute of Technology, Nagoya, Aichi, Japan}
\email{shusuke.otabe@nitech.ac.jp}

\date{\today}


\thanks{}

\begin{abstract}
In the present paper, we prove the retract rationality of the classifying spaces $BG$ for several types of finite connected group schemes $G$ over algebraically closed fields of positive characteristic $p>0$. In particular, we prove the retract rationality for the finite simple group schemes $G$ associated with the generalized Witt algebras $W(m;\bm{n})$ in the case when $\bm{n}=\bm{1}$ or $m=1$. To this end, we study the automorphism group schemes of the generalized Witt algebras and establish triangulations for them. Moreover, we extend the notion of Witt--Ree algebra to general base rings and discuss their properties. 
\end{abstract}

\maketitle

\thispagestyle{empty}



\section{Introduction}

In \cite{saltman}, Saltman provided the first counter-example for Noether's problem over \textit{algebraically closed} base fields. He proved that for any algebraically closed field $k$ of characteristic $p\ge 0$ and any prime number $\ell$ with $\ell\neq p$, there exists a finite group $G$ of order $\ell^9$ whose classifying space $B_kG$ is not retract rational. Although Noether's problem for finite groups generally has such a negative answer, it is usually challenging to determine the stable birational type of the classifying space $B_kG$ for a given finite group $G$. In fact, it is open whether or not any finite non-abelian simple group has stably rational classifying space. 

\begin{ques}(cf.\cite[Hypothesis 1.1]{BMP}\cite{Kunyavskii})
For any finite non-abelian simple group $G$, is the classifying space $B_{\mathbb{C}}G$ stably rational? 
\end{ques}

In the present paper, we discuss Noether's problem for finite {\it connected} group schemes in positive characteristic $p>0$ and prove the retract rationality of the classifying spaces in several cases. A negative answer to Noether's problem for finite connected group schemes is first given by Scavia in \cite{Scavia} over non-algebraically closed base fields. Up to the author's knowledge, no counter-example to Noether's problem is known for finite connected group schemes over {\it algebraically closed} fields. More precisely, the author does not know whether there exists a finite connected group scheme $G$ over an algebraically closed field of characteristic $p>0$ for which the classifying space $BG$ is not retract rational. As the motivation of the present work comes from this observation, we will stick to the case when the base field is algebraically closed. 
Let $k$ be an algebraically closed field of characteristic $p>0$. 
We will see that several basic types of finite connected group schemes $G$ over $k$ have {retract rational} classifying spaces $BG$. Although the general case where $G$ is an arbitrary connected group scheme of finite type (not necessarily finite) should be discussed (cf.\ \cite[Introduction]{Mer}), the present paper focuses on the finite case.  

In Section \ref{sec:trigonalizable}, we consider the case where $G$ is a finite solvable $k$-group scheme. As a main result, we will prove the following result, which says that the classifying space $BG$ is always retract rational for a large class of finite connected solvable group schemes.

\begin{thm}(cf.\,Theorem \ref{thm: ret rat trig})\label{thm-int: ret rat trig}
Let $G$ be a finite $k$-group scheme. If $G$ is trigonalizable, the associated classifying space $BG$ is retract rational over $k$. 
\end{thm}

In Section \ref{sec:Witt alg}, we will treat the problem in the case where $G$ is a finite connected non-abelian simple $k$-group scheme. As explained in \cite{Viviani}, the classification of finite connected non-abelian simple $k$-group schemes is equivalent to the classification of finite dimensional (non-abelian) simple Lie algebras over $k$. In characteristic $p>3$, a classification theory for finite dimensional simple Lie algebras over $k$ is established by Block--Wilson--Strade--Premet (cf.\ \cite{Strade}). For example, in characteristic $p>5$, the theory says that all the finite dimensional simple Lie algebras over $k$ are divided into two types, namely \textit{classical} type or \textit{Cartan} type. The classical type means that a simple Lie algebra is obtained by the modulo $p$ reduction of a complex simple Lie algebra. For example, $\mathfrak{sl}_n$ and $\mathfrak{sp}_{2n}$ over $k$ are classical type simple Lie algebras and the corresponding simple group schemes are nothing but the first Frobenius kernels ${\rm SL}_{n(1)}$ and ${\rm Sp}_{2n(1)}$ of the semi-simple simply connected algebraic groups ${\rm SL}_n$ and ${\rm Sp}_{2n}$. As these are {\it special} algebraic groups, the conclusion of Corollary \ref{cor:special G_r} implies that the associated simple group schemes ${\rm SL}_{n(1)}$ and ${\rm Sp}_{2n(1)}$ have retract rational classifying spaces.

On the other hand, in a large part of the present paper, we will deal with the most basic class of Cartan type simple Lie algebras, namely, the \textit{generalized Witt algebras} $W(m;\bm{n})$ for $m\in\Z_{>0}$ and $\bm{n}\in\Z_{>0}^{m}$.  We denote by $\Gamma(m;\bm{n})$ the finite simple $k$-group scheme associated with $W(m;\bm{n})$. 
As a main result, we will prove the following. 

\begin{thm}(cf.\,Theorem \ref{thm:ret rat BGamma(m,n)})\label{thm-int:ret rat BGamma(m,n)}
Suppose that $k$ is an algebraically closed field of characteristic $p>3$. If $\bm{n}=\bm{1}$ or $m=1$, then the classifying space $B\Gamma(m;\bm{n})$ of the finite simple group scheme $\Gamma(m;\bm{n})$ is retract rational over $k$. 
\end{thm}

A key ingredient of the proof is the following decomposition for the automorphism group schemes $G(m;\bm{n})\Def\uAut_k(W(m;\bm{n}))$ of the generalized Witt algebras $W(m;\bm{n})$. 

\begin{thm}(cf.\ Theorem \ref{thm:triangulation})\label{thm-int:triangulation}
For any positive integer $m>0$ and any $m$-tuple $\bm{n}=(n_1,\dots,n_m)\in\Z^m_{>0}$ of positive integers, there exists a system of closed subgroup schemes $G^{\bullet}\subset G(m;\bm{n})$ of the automorphism group scheme $G(m;\bm{n})$, where $\bullet\in\{-,0,+\}$, such that the following conditions are satisfied.
\begin{enumerate}
\renewcommand{\labelenumi}{(\roman{enumi})}
\item $G^-$ is isomorphic to the finite connected abelian unipotent $k$-group scheme $\prod_{i=1}^mW_{n_i(1)}$, where $W_{n_i(1)}$ is the Frobenius kernel of the Witt vector group $W_{n_i}$ of length $n_i$ for $1\le i\le m$.  

\item $G^0$ is isomorphic to the general linear algebraic group $\GL_m$.

\item $G^+$ is a smooth unipotent $k$-group scheme.

\item Furthermore, the multiplication map of $G(m;\bm{n})$ induces an isomorphism of $k$-schemes, 
\[
G^+\times G^0\times G^-\xrightarrow{~\simeq~}~G(m;\bm{n})~;~(\varphi^+,\varphi^0,\varphi^-)\mapsto \varphi^+\circ\varphi^0\circ\varphi^-.
\]
\end{enumerate}
\end{thm} 

In the case where $\bm{n}=\bm{1}$, there exists a natural isomorphism of $k$-group schemes 
\[
\uAut_k(k[y_1,\dots,y_m]/(y_1^p,\dots,y_m^p))\xrightarrow{~\simeq~}G(m;\bm{1})
\]
(cf.\cite{Waterhouse0}), and such decompositions for the automorphism group schemes of truncated polynomial rings $k[y_1,\dots,y_m]/(y_1^p,\dots,y_m^p)$ were investigated by Severitt in \cite{severitt}.

Thanks to Theorem \ref{thm-int:triangulation}, we can reduce the problem to the rationality problem for $G(m;\bm{n})$ (cf.\ Corollary \ref{cor:triangulation}). Moreover, by using a characterization of retract rationality in terms of lifting problems for torsors (cf.\ Proposition \ref{prop:ret rat}), the problem can be further reduced to solving a certain lifting problem of $G(m;\bm{n})$-torsors. Precisely, the proof of Theorem \ref{thm-int:ret rat BGamma(m,n)} is reduced to proving the next result.

\begin{thm}(cf.\,Proposition \ref{prop:ret rat BG(m,n)})
Suppose that $k$ is an algebraically closed field of characteristic $p>3$. Let $(R,\frkm)$ be the localization of a polynomial ring $k[X_1,\dots,X_N]$ at a prime ideal $\frkP\subset k[X_1,\dots,X_N]$ with residue field $\kappa\Def R/\frkm$ with $\Dim_{\kappa}\Omega^1_{\kappa/k}\ge n_1+\cdots+n_m$. If $\bm{n}=\bm{1}$ or $m=1$, the natural map 
\[
H^1_{\fppf}(R,G(m;\bm{n}))\lto H^1_{\fppf}(\kappa,G(m;\bm{n}))
\] 
is surjective. 
\end{thm}

To settle the lifting problem, we will make use of the classification theory due to Ree\cite{Ree}, Wilson\cite{Wilson} and Waterhouse\cite{Waterhouse} for twisted forms of the generalized Witt algebras $W(m;\bm{n})$. However, as most of their results about these twisted forms are proved in the case when the base ring $R$ is a field, we need to extend some of their arguments to more general base rings. In Section \ref{sec:Witt-Ree}, we will generalize the definition of \textit{Witt--Ree algebra} in the sense of Waterhouse\cite{Waterhouse} to obtain candidates of twisted forms of the generalized Witt algebras over any general base rings (cf.\ Definition \ref{def:Witt-Ree}).

\begin{ack}
The author would like to express his gratitude to Lei Zhang and Fabio Tonini for introducing him to the present topic, that is, rationality problem for finite connected group schemes. The author also would like to thank Takao Yamazaki for having fruitful discussions. This work was supported by JSPS KAKENHI Grant Numbers JP21K20334, JP24K16894.
\end{ack}


\section{Trigonalizable group schemes}\label{sec:trigonalizable}

\subsection{Preliminaries}\label{sec:pre}

In this subsection, we will briefly recall several basic background materials for rationality problems for group schemes. Let $k$ be a field and $G$ a finite $k$-group scheme. A finite dimensional $k$-linear representation $\rho\colon G\to\GL_V$ is said to be {\em generically free} if there exists a $G$-stable Zariski dense open  subset $U\subset V$ such that the action of  $G$ on $U$ is free, i.e.\ the morphism $U\times G\to U\times U;(u,g)\mapsto (u,ug)$ is a monomorphism. As $G$ is finite, one may take such $U$ so that it is affine (cf.\cite[Lemma 2.2]{Brion}), and $U/G:=\Spec H^0(U,\scrO_U)^G$ then gives a geometric quotient of $U$ under the action of $G$ so that the quotient map $U\to U/G$ is a $G$-torsor (cf.\cite[Chapter I, 5.5.(6)]{jan03}), i.e.\ the morphism $U\times G\to U\times_{U/G}U$ is an isomorphism of $k$-schemes. The $G$-torsor $U\to U/G$ is called a {\it standard $G$-torsor}. By the no-name lemma, the stably birational equivalence class of $U/G$ does not depend on the choice of $\rho$ and $U$ (cf.\cite[Section 4.1]{Mer}). Henceforth, we put $BG\Def U/G$ and call it the {\em classifying space of $G$}. Note that $BG$ is a smooth connected scheme over $k$. A similar construction of $BG$ for an arbitrary affine algebraic $k$-group scheme is well-understood by the experts (cf.\ \cite[Section 4]{BF}\cite{BK}\cite[Section 4.1]{Mer}\cite{Scavia}). 

\begin{definition}(cf.\cite[Section 3]{Mer})\label{def:ret rat} 
A smooth connected scheme $X$ over $k$ is said to be \textit{retract rational} if there exists a dominant rational map from an affine space $f\colon \A^n_k\dashrightarrow X$ for some $n$ which admits a section $g\colon  W\to \A_k^n$ defined over some dense open  subset $W$ of $X$.
\end{definition}

\begin{rem}\label{rem:def ret rat}
The condition that $X$ is retract rational depends only on the stably birational equivalence class of $X$. In other words, $X$ is retract rational if and only if so is $X\times\A^1_k$ (cf.\,\cite[Lemma 3.2]{Mer}). In particular, any stably rational varieties are retract rational (cf.\,\cite[Proposition 3.4]{Mer}). Moreover, for any affine algebraic $k$-group scheme $G$, the condition that $BG$ is retract rational does not depend on either the choice of generically free representation $\rho\colon G\to\GL_V$ or the choice of a $G$-stable dense open subset $U$. 
\end{rem}

In the present paper, we would like to ask whether $BG$ is retract rational for several group schemes $G$. 
To approach this question, we will frequently use the following characterization of  the retract rationality in terms of lifting problems for torsors.

\begin{prop}(cf.\ \cite[Propositions 3.1 and 4.2]{Mer})\label{prop:ret rat}
Let $X$ be a smooth connected scheme over $k$. Then the following are equivalent.
\begin{enumerate}
\renewcommand{\labelenumi}{(\alph{enumi})}
\item $X$ is retract rational over $k$. 

\item For any local $k$-algebra $(R,\mathfrak{m})$, there exists a dense open subset $W\subset X$ such that the map $W(R)\to W(R/\mathfrak{m})$ is surjective. 

\item For any local $k$-algebra $(R,\mathfrak{m})$ with $R/\mathfrak{m}\simeq k(X)$, there exists a dense open subset $W\subset X$ such that the map $W(R)\to W(R/\frkm)$ is surjective. 

\item For any $k$-algebra homomorphism from any polynomial ring $\alpha\colon k[x_1,\dots,x_n]\to k(X)$ with $\Frac(\im(\alpha))=k(X)$ and with 
$\mathfrak{P}\Def\Ker(\alpha)$, there exists a dense open subset $W\subset X$ such that the map $W(k[x_1,\dots,x_n]_{\mathfrak{P}})\to W(k(X))$ is surjective. 
\end{enumerate}
Moreover, in the case where $X=BG$ is the classifying space for some $G$, these are also equivalent to each of the conditions below. 
\begin{itemize}
\item[(e)] For any local $k$-algebra $(R,\mathfrak{m})$ with $R/\frkm$ an infinite field, the map $H^1_{\fppf}(R,G)\to H^1_{\fppf}(R/\mathfrak{m},G)$ is surjective. 

\item[(f)] For any local $k$-algebra $(R,\mathfrak{m})$ with $R/\frkm\simeq k(X)$, the map $H^1_{\fppf}(R,G)\to H^1_{\fppf}(R/\mathfrak{m},G)$ is surjective. 

\item[(g)] For any $k$-algebra homomorphism from any polynomial ring $\alpha\colon k[x_1,\dots,x_n]\to k(X)$ with $\Frac(\im(\alpha))=k(X)$ and with  $\mathfrak{P}\Def\Ker(\alpha)$, the map $H^1_{\fppf}\left(k[x_1,\dots,x_n]_{\mathfrak{P}},G\right)\to H^1_{\fppf}(k(X),G)$ is surjective. 
\end{itemize}
\end{prop}


\subsection{Frobenius kernels of special algebraic groups}\label{sec:special}

In this subsection, we will prove the retract rationality of $BG$ for the Frobenius kernels 
\[
G=\Sigma_{(r)}\Def\Ker\left(F^{(r)}\colon \Sigma\lto\Sigma^{(p^r)}\right)\quad (r\in\Z_{>0})
\] 
of any {\it special} algebraic group $\Sigma$ over an algebraically closed field $k=\overline{k}$ of characteristic $p>0$. 

\begin{definition}(cf.\cite[Theorem 1.1]{RT20})\label{def:special}
For an affine group scheme $G$ of finite type over a field $k$, the following are equivalent.
\begin{enumerate}
\renewcommand{\labelenumi}{(\alph{enumi})}
\item Any fppf $G$-torsor $T\to X$ over a reduced separated scheme $X$ of finite type over $k$ is locally trivial in the Zariski topology. 

\item $H^1_{\fppf}(K,G)=1$ for any field $K$ over $k$. 

\item $H^1_{\fppf}(R,G)=1$ for any local $k$-algebra $R$. 

\item $H^1_{\fppf}(S,G)=1$ for any semi-local $k$-algebra $S$. 
\end{enumerate}
The $k$-group scheme $G$ is said to be \textit{special} if it satisfies these conditions. 
\end{definition}

\begin{rem}(cf.\cite[Proposition 2.3, Lemma 3.1]{RT20})\label{rem:special}
\begin{enumerate}
\item If $G$ is special over a field $k$, then it is smooth and connected over $k$.
\item Suppose that $1\to G'\to G\to G''\to 1$ is an exact sequence of group schemes of finite type over a field $k$. If $G'$ and $G''$ are special, then so is $G$. 
\end{enumerate} 
\end{rem}

\begin{ex}(cf.\cite[Lemma 3.2]{RT20})\label{ex:special}
\begin{enumerate}
\renewcommand{\labelenumi}{(\arabic{enumi})}
\item The additive group $\G_{a,k}$ and the multiplicative scheme $\G_{m,k}$ are special. Hence, by Remark \ref{rem:special}(2), so are smooth connected $k$-split solvable group schemes over $k$.   

\item For any positive integer $n$, the algebraic groups $\GL_n$, $\SL_n$ and ${\rm Sp}_{2n}$ are special. 
\end{enumerate}
\end{ex}

\begin{rem}\label{rem:special implies ret rat} 
If $G$ is special over $k$, the classifying space $BG$ is retract rational. This is immediate from Definition \ref{def:special} together with Proposition \ref{prop:ret rat}. 
\end{rem}

\begin{prop}\label{prop:special}\label{prop:special G''}
Suppose that $1\to G'\to G\to G''\to 1$ is an exact sequence of affine group schemes of finite type over a field $k$. 
If $G''$ is special and it is rational as a variety over $k$, then $BG$ and $BG'$ are stably birationally equivalent to each other. 
\end{prop}

\begin{proof}
Let $U\to U/G$ be a standard $G$-torsor. Then the partial quotient $U\to U/G'$ is a standard $G'$-torsor. Note that the natural map $U/G'\to U/G$ is a $G''$-torsor. However, as $G''$ is special, the $G''$-torsor $U/G'\to U/G$ is locally trivial in the Zariski topology, hence $U/G'$ is birationally equivalent to $(U/G)\times G''$. As $G''$ is rational as a variety, $(U/G)\times G''$ is stably birationally equivalent to $U/G$. Thus, it follows that $U/G'$ is stably birationally equivalent to $U/G$. This completes the proof. 
\end{proof}

\begin{cor}\label{cor:special G_r}
Let $k=\overline{k}$ be an algebraically closed field of characteristic $p>0$. If $G$ is special over $k$, then for any integer $r>0$, the classifying space $B G_{(r)}$ for the $r$-th Frobenius kernel $G_{(r)}$ is retract rational. 
\end{cor}

\begin{proof}
We have the exact sequence of affine $k$-group schemes of finite type 
\[
1\lto G_{(r)}\lto G\xrightarrow{~F^r~} G^{(p^r)}\lto 1.
\] 
As $k$ is algebraically closed, the smooth affine group scheme $G^{(p^r)}$ is rational as a variety over $k$ (cf.\cite[14.14 Remark]{Borel}). Moreover, by the assumption, $BG$ and $BG^{(p^r)}$ are retract rational (cf. Remark \ref{rem:special implies ret rat}). By Proposition \ref{prop:special G''}, $BG_{(r)}$ is stably birationally equivalent to $BG$. Therefore, $BG_{(r)}$ is retract rational as desired. 
\end{proof}

\begin{rem}
Let $G$ be an arbitrary smooth connected affine group scheme over an algebraically closed field $k=\overline{k}$ of characteristic $p>0$. Then for any integer $r>0$, the classifying space $B G_{(r)}$ of the $r$-th Frobenius kernel $G_{(r)}$ is separably unirational. 
Indeed, let $U\to U/G$ be a standard $G$-torsor. Then the partial quotient $U\to U/G_{(r)}$ gives a standard $G_{(r)}$-torsor. We have to show that $U/G_{(r)}$ is separably unirational. Note that the natural map $U/G_{(r)}\to U/G$ is a $G^{(p^r)}$-torsor. Consider the pull-back $X\to U$ of the $G^{(p^r)}$-torsor along the standard $G$-torsor $U\to U/G$,
\begin{equation*}
\begin{xy}
\xymatrix{\ar@{}[dr]|\square
X\ar[r]\ar[d]&U/G_{(r)}\ar[d]\\
U\ar[r]&U/G.
}
\end{xy}
\end{equation*}  
As $U\to U/G$ factors through $U/G_{(r)}$, the $G^{(p^r)}$-torsor $X\to U$ admits a section, hence it is trivial, i.e.\ $X\simeq U\times G^{(p^r)}$. As $G^{(p^r)}$ is a smooth connected algebraic group scheme over an algebraically closed field $k$, it is rational over $k$ (cf.\cite[14.14 Remark]{Borel}). Therefore, $X$ is rational over $k$. As the projection $X\to U/G_{(r)}$ is a smooth dominant morphism, this implies that $U/G_{(r)}$ is separably unirational. 
\end{rem}

\subsection{Trigonalizable group schemes}\label{sec:trig}

In this subsection, we will discuss the retract rationality of $BG$ for finite {\it trigonalizable} group schemes $G$ in positive characteristic $p>0$. 

\begin{definition}(cf.\cite[IV,\S2, 3.1]{DG}\cite[Definition 1.3]{TV13})
An affine group scheme $G$ of finite type over a field $k$ is said to be \textit{trigonalizable} if it has a normal unipotent subgroup scheme $G_{\rm uni}$ such that $G/G_{\rm uni}$ is diagonalizable. 
\end{definition}

\begin{thm}\label{thm: ret rat trig}
Let $G$ be a finite group scheme over a perfect field $k$ of characteristic $p>0$. If $G$ is trigonalizable,  the classifying space $BG$ is retract rational.  
\end{thm}

\begin{proof}
Let $G_{\rm uni}\triangleleft\, G$ be a normal unipotent subgroup such that $G/G_{\rm uni}$ is  diagonalizable. There exists an isomorphism of $k$-group schemes $G/G_{\rm uni}\simeq\mathrm{Diag}_k(M)$ for some finite abelian group $M$. By the Kummer theory, $\mathrm{Diag}_k(M)$ is the kernel of a morphism $\G_m^n\to\G_m^n$ for some $n$, Proposition \ref{prop:special G''} implies that the classifying space $B(G/G_{\rm uni})\simeq B(\mathrm{Diag}_k(M))$ is retract rational. Therefore, the theorem is a consequence of the next lemma.  
\end{proof}

\begin{lem}\label{lem: trig}
Suppose that $1\to G'\to G\to G''\to 1$ is an exact sequence of affine group schemes of finite type over a perfect field $k$ of characteristic $p>0$ which satisfies the following conditions. 
\begin{enumerate}
\renewcommand{\labelenumi}{(\roman{enumi})}
\item $G'$ is finite and unipotent.  
\item $BG''$ is retract rational.
\end{enumerate}
Then the classifying space $BG$ is retract rational. 
\end{lem}

\begin{proof}
By condition (i), the group scheme $G'$ is a solvable group scheme. If we define the subgroups $G'^{n}<G'$ inductively as follows
\[
G'^0\Def G',\quad G'^n\Def [G'^{n-1},G'^{n-1}]\quad (n\ge 1),
\] 
we have $G'^N=1$ for sufficiently large $N>0$, and we obtain a descending normal series of subgroup schemes
\[
1=G'^N\le G'^{N-1}\le\cdots\le G'^1\le G'^0=G'
\] 
such that the successive quotients $G'^{n-1}/G'^n$ are abelian group schemes. As the commutator subgroup scheme $G'^1$ is a characteristic subgroup scheme of $G'$, it is a  normal subgroup scheme in $G$. Hence, we obtain the following commutative diagram of short exact sequences.
\[
\begin{xy}
\xymatrix{
&1\ar[d]&1\ar[d]&&&\\
&G'^1\ar@{=}[r]\ar[d]&G'^1\ar[d]&&\\
1\ar[r]&G'\ar[r]\ar[d]&G\ar[r]\ar[d]&G''\ar[r]\ar@{=}[d]&1\\
1\ar[r]&G'/G'^1\ar[r]\ar[d]&G/G'^1\ar[r]\ar[d]&G''\ar[r]&1\\
&1&1&&&
}
\end{xy}
\]
This, together with induction on the order of $G'$, allows us to reduce the problem to the case where $G'$ is abelian. Indeed, by applying the lemma to the exact sequence 
\[
1\to G'/G'^1\to G/G'^1\to G''\to 1
\] 
where $G'/G'^1$ is abelian, we can conclude that $B(G/G'^1)$ is retract rational. Hence, it suffices to show the lemma for the exact sequence $1\to G'^1\to G\to G/G'^1\to 1$. However, as the order of $G'^1$ is strictly smaller than the order of $G'$, the induction hypothesis tells us that $BG$ is retract rational. Therefore, we are reduced to the case when $G'$ is abelian.

Suppose that $G'$ is a finite, abelian and unipotent group scheme over $k$. Let us denote by $V\colon G'^{(p)}\to G'$ the Verschiebung for $G'$, which is the Cartier dual $V\Def F^D$ to the relative Frobenius $F\colon G'^D\to {G'^{D}}^{(p)}$. As $G'$ is unipotent, its dual $G'^D$ is a local group scheme and we have $V^N=0$ on $G'$ for sufficiently large $N>0$ (cf.\cite[Section 1]{Tossici}). Thus, we obtain a series of subgroup schemes
\[
1=V^N(G'^{(p^N)})\le V^{N-1}(G'^{(p^{N-1})})\le\cdots\le V(G'^{(p)})\le G',
\]
where the successive quotients $V^{n-1}(G'^{(p^n-1)})/V^n(G'^{(p^n)})$ are dual to group schemes of height at most one $\Ker(F\colon F^{n-1}({G'^D}^{(p^n-1)})\to F^n({G'^D}^{(p^n)}))$. As the Frobenius, hence the Verschiebung, commutes with any homomorphism of commutative group schemes, the image of Verschiebung is a characteristic subgroup scheme of $G'$. Hence, we obtain the following commutative diagram of short exact sequences.
\[
\begin{xy}
\xymatrix{
&1\ar[d]&1\ar[d]&&&\\
&V(G'^{(p)})\ar@{=}[r]\ar[d]&V(G'^{(p)})\ar[d]&&\\
1\ar[r]&G'\ar[r]\ar[d]&G\ar[r]\ar[d]&G''\ar[r]\ar@{=}[d]&1\\
1\ar[r]&G'/V(G'^{(p)})\ar[r]\ar[d]&G/V(G'^{(p)})\ar[r]\ar[d]&G''\ar[r]&1\\
&1&1&&&
}
\end{xy}
\]
Thus, by induction on the order of $G'$ as in the previous paragraph, the problem is reduced to the case where $V=0$ on $G'$, or equivalently to the case where $G'^D$ is of height one.

Suppose that $G'$ is finite, abelian and unipotent such that the Cartier dual $G'^D$ is of height one. Fix a local $k$-algebra $(R,\fm)$ and put $X=\Spec R$ and $X_0=\Spec R/\fm$. It suffices to show that the reduction map
\begin{equation*}
H^1_{\fppf}(X,G)\lto H^1_{\fppf}(X_0,G)
\end{equation*}
is surjective (cf.\ Proposition \ref{prop:ret rat}). The idea of the following argument comes from the work of Tossici--Vistoli for the essential dimension of trigonalizable group schemes \cite[Theorem 1.4]{TV13}. 
Fix a $G$-torsor $P_0\to X_0$ and set $P''_0\Def P_0/G'$, which is a $G''$-torsor over $X_0$. As $B G''$ is retract rational, $P''_0$ can be lifted to a $G''$-torsor $P''\to X$. Let $\cB G$ (respectively $\cB G''$) be the classifying stack of $G$ (respectively $G''$). Then the natural map $\cB G\to\cB G''$ is an fppf gerbe which is banded by the abelian group scheme $G'$. If we define the fppf gerbe $\mathcal{G}$ over $X$ by
\begin{equation*}
\mathcal{G}\Def\cB G\times_{\cB G'',P''}X.
\end{equation*}
Then we have 
\begin{equation*}
\calG_0\Def\calG\times_X X_0=\cB G\times_{\cB G'',P''_0}X_0. 
\end{equation*}
By definition, the gerbe $\calG\to X$ classifies $G$-torsors over $X$ which are liftings of $P''$.  Similarly, $\calG_0$ classifying the liftings of $P''_0$. The fixed $G$-torsor $P_0\to X_0$ then defines a section $X_0\to\calG_0$. Therefore, the surjectivity of the map $H^1_{\fppf}(X,G)\to H^1_{\fppf}(X_0,G)$ is reduced to showing that the reduction map
\begin{equation*}
\calG(X)\lto\calG_0(X_0)
\end{equation*}
is essentially surjective. For the essential surjectivity, note that the gerbe $\calG\to X$ is banded by the finite flat abelian unipotent $X$-group scheme $N=G'\wedge^{G''} P''$ whose Cartier dual is of height one. Such an $N$ admits a resolution by locally free sheaves over $X=\Spec R$ as follows (cf.\cite{Milne}),
\begin{equation}\label{eq:omega_N^D}
0\lto N\lto \omega_{N^D}\lto \omega_{N^D}^{(p)}\lto 0. 
\end{equation}
As $X=\Spec R$ is the spectrum of a local ring, we have $\omega_{N^D}\simeq\G_{a,X}^{n}$ for some $n>0$. In particular, we have  $H^2_{\fppf}(X,N)=0$, hence $\calG\simeq\cB N$. Therefore, it suffices to show that the reduction map
\begin{equation*}
H^1_{\fppf}(X,N)\lto H^1_{\fppf}(X_0,N_0)
\end{equation*}
is surjective, where $N_0\Def N\times_X X_0$. However, again by using the resolution (\ref{eq:omega_N^D}), this is immediate from the commutativity of the diagram
\begin{equation*}
\begin{xy}
\xymatrix{
R^n\ar@{->>}[r]\ar@{->>}[d]&H^1_{\fppf}(X,N)\ar[d]\\
(R/\fm)^n\ar@{->>}[r]&H^1_{\fppf}(X_0,N_0).
}
\end{xy}
\end{equation*}
This completes the proof. 
\end{proof}


\section{Simple group schemes associated with the generalized Witt algebras}\label{sec:Witt alg}

\subsection{The automorphisms of the generalized Witt algebras}\label{sec:autom}

In this subsection, we will recall the definition of the \textit{generalized Witt algebras} $W(m;\bm{n})$ and introduce the automorphism group schemes  $G(m;\bm{n})$ of them and their first Frobenius kernels $\Gamma(m;\bm{n})$. 
Moreover, toward the decomposition theorem for $G(m;\bm{n})$ (cf.\ Theorem \ref{thm:triangulation}), we will classify automorphisms of the generalized Witt algebras into three types. 

Let $k=\overline{k}$ be an algebraically closed field of characteristic $p>0$. Let $m\ge 1$ be a positive integer. Let $\Z_{\ge 0}^m$ be the set of $m$-tuples $\alpha=(\alpha_1,\dots,\alpha_m)$ of non-negative integers $\alpha_1,\dots,\alpha_m\ge 0$. We put $\bm{1}\Def (1,\dots,1)\in\Z_{\ge 0}^m$. We denote by $\widetilde{A}(m)\Def\Z\langle X^{(\alpha)}\,|\,\alpha\in\Z_{\ge 0}^m\rangle$ the divided power polynomial ring of variables $X_1,\dots,X_m$ with coefficients in $\Z$. Namely, $\widetilde{A}(m)$ is the $\Z$-subalgebra of the polynomial ring $\Q[X_1,\dots,X_m]$ which is generated as an additive subgroup by all the monomials $X^{(\alpha)}$ of the form
\[
X^{(\alpha)}=\dfrac{X_1^{\alpha_1}\cdots X_m^{\alpha_m}}{\alpha_1!\cdots\alpha_m!},\quad \alpha=(\alpha_1,\dots,\alpha_m)\in\Z_{\ge 0}^m.
\]
We define the $k$-algebra $A(m)$ to be $A(m):=\widetilde{A}(m)\otimes_{\Z}k$ and we set $x^{(\alpha)}:=X^{(\alpha)}\otimes 1\in A(m)$ for any $\alpha\in\Z_{\ge 0}^m$. Moreover, for any $m$-tuple $\bm{n}=(n_1,\dots,n_m)\in\Z^m_{>0}$ of positive integers $n_1,\dots,n_m\ge 1$, we define the $k$-subalgebra $A(m;\bm{n})\subset A(m)$ to be the one generated as a $k$-vector subspace by $x^{(\alpha)}$ with $\alpha_i<p^{n_i}$ for all $i=1,\dots,m$, i.e.\
\[
A(m;\bm{n})\Def k\bigl\langle x^{(\alpha)}\,\bigl |\,\alpha=(\alpha_1,\dots,\alpha_m)\in\Z^m_{\ge 0}~\text{with}~\alpha_i<p^{n_i}\bigl\rangle\subset A(m).
\] 
If we define the $k$-algebra map 
\begin{equation}\label{eq:aug}
\varepsilon \colon A(m;\bm{n})\lto k
\end{equation} 
by putting $\varepsilon(x^{(\alpha)})=0$ for any $\alpha\in\Z_{\ge 0}^m$ with $\alpha\neq(0,\dots,0)$, then the ideal $I:=\Ker(\varepsilon)\subset A(m;\bm{n})$ is the maximal ideal with $I^p=0$. Furthermore, the ideal $I$ admits a unique divided power structure $\{I\to I;f\mapsto f^{(r)}\}_{r=1}^{\infty}$ satisfying 
\[
x_i^{(r)}=x^{(r\delta_{i1},\dots,r\delta_{im})}=\dfrac{X_i^r}{r!}\otimes 1\quad(1\le i\le m,~0\le r<p^{n_i})
\]
where $\delta_{ij}$ is the Kronecker delta. If the $p$-adic expansion of $r\ge 0$ is given by
\[
r=r_0+r_1p+\cdots+r_{n_i-1}p^{n_i-1}\quad (0\le r_s<p),
\] 
we have 
\[
x_i^{(r)}=c_{ir}\prod_{s=0}^{n_i-1}(x_i^{(p^s)})^{r_s}
\]
for some $c_{ij}\in k^*$, the mapping $y_{is}\mapsto x_i^{(p^s)}$ defines an isomorphism of $k$-algebras $B(m;\bm{n})\xrightarrow{\simeq}A(m;\bm{n})$ from the truncated polynomial ring $B(m;\bm{n})$ onto $A(m;\bm{n})$, where $B(m;\bm{n})$ is defined to be
\[
B(m;\bm{n})\Def \dfrac{k[y_{is}\,|\,1\le i\le m,0\le s\le n_i-1]}{(y_{is}^p\,|\,1\le i\le m,0\le s\le n_i-1)}.
\]

A $k$-derivation $D\colon A(m;\bm{n})\to A(m;\bm{n})$ is called \textit{special} if it satisfies the condition that
\[
D(f^{(r)})=f^{(r-1)}D(f)\quad (f\in I,r\ge 1).
\]
We define the {\it generalized Witt} algebra $W(m;\bm{n})$ over $k$ to be the Lie subalgebra over $k$ of the derivation algebra $\Der_k(A(m;\bm{n}))$ consisting of special derivations $D$ on $A(m;\bm{n})$. For any $1\le i\le m$, we denote by $\partial_i\in W(m;\bm{n})$ the special derivation satisfying 
\[
\partial_i(x_j)=\delta_{ij}\quad (1\le j\le m).
\]
Then it turns out that $W(m;\bm{n})$ is a free $A(m;\bm{n})$-module with free basis $\{\partial_1,\dots,\partial_m\}$, i.e.
\[
W(m;\bm{n})=\sum_{i=1}^mA(m;\bm{n})\partial_i\subseteq\Der_k(A(m;\bm{n}))
\]
(cf.\cite{Strade}). The generalized Witt algebra $W(m;\bm{n})$ is a finite-dimensional simple Lie algebra over $k$ except in the case where $(p,m,\bm{n})=(2,1,1)$. In the case where $(p,m,\bm{n})\neq (2,1,1)$, the adjoint map $ad\colon W(m;\bm{n})\hookrightarrow\Der(W(m;\bm{n}))$ is bijective if and only if $\bm{n}=\bm{1}$, in only which case the generalized Witt algebra admits a restricted structure $D\mapsto D^{[p]}$. For arbitrary $\bm{n}\in\Z_{>0}^m$, the $p$-envelope $W(m;\bm{n})_{[p]}$ of the generalized Witt algebra $W(m;\bm{n})$ in $\Der_k(W(m;\bm{n}))$ coincides with the Lie algebra $\Der_k(W(m;\bm{n}))$ of $k$-deviations on $W(m;\bm{n})$, i.e.\,
\begin{equation}
W(m;\bm{n})_{[p]}=\Der_k(W(m;\bm{n}))\simeq W(m;\bm{n})\oplus\sum_{i=1}^m\sum_{s=1}^{n_i-1}k\partial_i^{p^s},
\end{equation}
where the last isomorphism is given by the adjoint representation (cf.\cite[Lemma 4]{Wilson}\cite[p.358, Theorem 7.1.2]{Strade}). Henceforth, we consider the derivation algebra $\Der_k(W(m;\bm{n}))$ as a subalgebra of $\Der_k(A(m;\bm{n}))$ as follows,
\begin{equation}
\Der_k(W(m;\bm{n}))=W(m;\bm{n})\oplus\sum_{i=1}^m\sum_{s=1}^{n_i-1}k\partial_i^{p^s}\subset\Der_k(A(m;\bm{n})).
\end{equation}

We define the affine $k$-group scheme of finite type $G(m;\bm{n})$ to be the automorphism group scheme $G(m;\bm{n})\Def\uAut_k(W(m;\bm{n}))$ of the Lie algebra $W(m;\bm{n})$ over $k$. Its Lie algebra is canonically isomorphic to the derivation algebra $\Der_k(W(m;\bm{n}))$ of the Lie algebra $W(m;\bm{n})$, i.e.\,
\begin{equation}
\Lie(G(m;\bm{n}))\simeq\Der_k(W(m;\bm{n}))
\end{equation} 
(cf.\,\cite[Lemma 2.6]{GP}\cite[Section 1]{Milne2022}). Therefore, the first Frobenius kernel 
\[
\Gamma(m;\bm{n})\Def\Ker\left(F\colon G(m;\bm{n})\lto G(m;\bm{n})^{(p)}\right)
\] 
is the finite group scheme of height one associated with the restricted Lie algebra $\Der_k(W(m;\bm{n}))\simeq W(m;\bm{n})_{[p]}$. According to \cite{Viviani}, the group scheme $\Gamma(m;\bm{n})$ is a finite connected non-abelian simple group scheme over $k$ unless $(p,m,\bm{n})=(2,1,1)$.

\begin{definition}(cf.\cite[p.187, Definition]{Waterhouse})
Let $R$ be a $k$-algebra. An $R$-linear automorphism of the algebra $\varphi\in\uAut_k(A(m;\bm{n}))(R)=\Aut_R(A(m;\bm{n})\otimes_k R)$ is said to be a \textit{derivation-automorphism} with respect to $W(m;\bm{n})$ if $\varphi_*(D)=\varphi\circ D\circ \varphi^{-1}\in W(m;\bm{n})\otimes_k R$ for any $D\in W(m;\bm{n})\otimes_k R$. We denote by $\uAut_k(A(m;\bm{n}),W(m;\bm{n}))$ the group scheme of derivation-automorphisms of $A(m;\bm{n})$. 
\end{definition}

\begin{thm}(cf.\cite[Theorem C]{Waterhouse})\label{thm:Aut W} Suppose that $p>3$. Then the natural homomorphism of affine $k$-group schemes $\uAut_k(A(m;\bm{n}),W(m;\bm{n}))\to G(m;\bm{n});\varphi\mapsto\varphi_*$ is an isomorphism of $k$-group schemes. 
\end{thm}

In particular, any automorphism of the generalized Witt algebra $W(m;\bm{n})\otimes_k R$ over any $k$-algebra $R$ is induced by an automorphism of the algebra $A(m;\bm{n})\otimes_k R$.

Now we introduce three subgroup schemes $G^{\bullet}\subset G(m;\bm{n})$ for $\bullet\in\{-,0,+\}$ (cf.\ Definitions \ref{def:G^-}, \ref{def:G^0} and \ref{def:G^+}). We begin with the subgroup scheme $G^-$. For any positive integer $n\ge 1$, we denote by $W_n$ the algebraic group of Witt vectors of length $n$ over $k$. This is a connected smooth abelian unipotent $k$-group scheme whose underlying scheme is isomorphic to the affine space $\A^n_k=\Spec k[t_0,t_1,\dots,t_{n-1}]$. In particular, for any $k$-algebra $R$, we can identify $W_n(R)=R^n$.

Via the natural embeddings
\[
\Lie(W_{n_i})\hookrightarrow\Der_k(A(m;\bm{n}))~;~\sum_{s=0}^{n_i-1}a_i\dfrac{d}{dt_{s}}\mapsto \sum_{s=0}^{n_i-1}a_i\partial_i^{p^s},
\]
we obtain the identification 
\[
\bigoplus_{i=1}^m\Lie(W_{n_i})=\sum_{i=1}^m\sum_{s=0}^{n_i-1}a_i\partial_i^{p^s}\subset\Der_k(A(m;\bm{n})). 
\]
This gives rise to an embedding of $k$-algebraic groups
\[
\prod_{i=1}^mW_{n_i}\hookrightarrow\GL_{A(m;\bm{n})},
\]
whose restriction to the first Frobenius kernel $\prod_{i=1}^mW_{n_i(1)}$ factors through the subgroup scheme $\Gamma(m;\bm{n})\subset\GL_{A(m;\bm{n})}$. Hence we have a natural embedding of $k$-group schemes
\begin{equation}\label{eq:def of G^-}
\prod_{i=1}^mW_{n_i(1)}\hookrightarrow G(m;\bm{n}).
\end{equation}
More precisely, each factor $W_{n_i}\hookrightarrow\GL_{A(m;\bm{n})}$ of the embedding can be described in terms of the Artin--Hasse exponential series $E_p(t)\in k[[t]]$ as follows,
\[
W_{n_i}\hookrightarrow\GL_{A(m;\bm{n})}~;~\bm{a}=(a_0,a_1,\dots,a_{n_i-1})\mapsto E_p(a_0\partial_i)E_p(a_1\partial_i^p)\cdots E_p(a_{n_i-1}\partial_i^{n_i-1}).
\]

\begin{definition}\label{def:G^-}
We denote by $G^-(m;\bm{n})$ (or by $G^-$ shortly) the image of the homomorphism (\ref{eq:def of G^-}). 
Let $R$ be a $k$-algebra. 
Moreover, for any 
\[
\bm{a}=(a_{is})_{1\le i\le m,0\le s\le n_i-1}\in \prod_{i=1}^mW_{n_i(1)}(R),
\]
we denote by $\varphi(\bm{a})$ the derivation-automorphism of $A(m;\bm{n})\otimes_k R$ such that $\varphi(\bm{a})_*$ coincides with the image of $\bm{a}$ in $G(m;\bm{n})$. 
\end{definition}

\begin{rem}\label{rem:AH p-torsion}
Note that
\[
E_p(t)\equiv 1+t+\dfrac{t^2}{2!}+\cdots+\dfrac{t^{p-1}}{(p-1)!}\mod t^pk[[t]].
\]
Hence, for any $\bm{a}=(a_0,a_1,\dots,a_{n_i-1})\in W_{n_i(1)}(R)$, we have
\[
E_p(a_s\partial_i^{p^s})=1+a_s\partial_i^{p^s}+\dfrac{a_s^2\partial_i^{2p^s}}{2!}+\cdots+\dfrac{a_s^{p-1}\partial_i^{(p-1)p^s}}{(p-1)!}.
\]
Moreover, as 
\begin{align*}
E_p(t)E_p(-t)&\equiv\left(1+t+\dfrac{t^2}{2!}+\cdots+\dfrac{t^{p-1}}{(p-1)!}\right)\left(1+(-t)+\dfrac{(-t)^2}{2!}+\cdots+\dfrac{(-t)^{p-1}}{(p-1)!}\right)\\
&=\sum_{k=0}^{p-1}\sum_{i=0}^k\dfrac{t^i}{i!}\dfrac{(-t)^{k-i}}{(k-i)!}=\sum_{k=0}^{p-1}\dfrac{\{t+(-t)\}^k}{k!}=1\mod t^pk[[t]],
\end{align*}
for any $\bm{a}=(a_{is})_{1\le i\le m,0\le s\le n_i-1}\in \prod_{i=1}^mW_{n_i(1)}(R)$, 
we have $\varphi(\bm{a})^{-1}=\varphi(-\bm{a})$. 
\end{rem}

\begin{rem}\label{rem:G^-}
Suppose that $m=1$ and write $\partial\Def\partial_1$. For any 
$a\in R$ with $a^p=0$, 
we define $(a+x)^{(r)}$ to be $(a+x)^{(r)}\Def\dfrac{(a+x)^r}{r!}$ for $1\le r<p$, and 
\[
(a+x)^{(r)}\Def\sum_{i=0}^{p-1}a^{(i)}x^{(r-i)}=\left(\sum_{i=0}^{p-1}\dfrac{a^{i}\partial^i}{i!}\right)(x^{(r)})=E_p(a\partial)(x^{(r)})
\]
for $r\ge p$. 
Moreover, for any $s,t\ge 0$, we have $E_p(a\partial^{p^t})(x^{(p^s)})=x^{(p^s)}$ for  $s<t$ and $E_p(a\partial^{p^s})(x^{(p^s)})=a+x^{(p^s)}$ for $s=t$. By noticing that 
\begin{align*}
\dfrac{(p^ij)!}{j!(p^i!)^{j}}&=\prod_{l=0}^{j-1}\dfrac{((l+1)p^i-1)!}{(lp^i)!(p^i-1)!}=\prod_{l=0}^{j-1}\dfrac{((l+1)p^i-1)((l+1)p^i-2)\cdots ((l+1)p^i-(p^i-1))}{(p^i-1)!}\\&\equiv \prod_{l=0}^{j-1}\dfrac{(-1)^{p^i-1}(p^i-1)!}{(p^i-1)!}=1\mod p\quad (i,j\ge 0),
\end{align*}
we have
\[
E_p(a\partial^{p^t})(x^{(p^s)})=\sum_{i=0}^{p-1}a^{(i)}x^{(p^s-ip^t)}=\sum_{i=0}^{p-1}a^{(i)}(x^{(p^t)})^{(p^{s-t}-i)}
\]
for $s\ge t$, which we denote by $(a+x^{(p^t)})^{(p^{s-t})}\Def E_p(a\partial^{p^t})(x^{(p^s)})$.  
These computations give a description of $E_p(a\partial_i^{p^t})(x_i^{(p^s)})$ even in the case where $m$ is arbitrary as follows.
\begin{align}\label{eq:E_p x_i}
E_p(a\partial_i^{p^t})(x_j^{(p^s)})=
\begin{cases}
x_j^{(p^s)},&i\neq j,\\
x_i^{(p^s)},&i=j,s<t,\\
(x_i^{(p^t)}+a)^{(p^{s-t})},&i=j,s\ge t
\end{cases}
\end{align} 
\end{rem}

Next we define the subgroup schemes $G^{\bullet}\subset G(m;\bm{n})$ for $\bullet\in\{0,+\}$. To this end, we prove the following lemma. 

\begin{lem}\label{lem:def of G^+,0} 
Let $R$ be a $k$-algebra. 
An automorphism of $R$-algebra $\varphi\in\uAut(A(m;\bm{n}))(R)$ with $\varphi(I\otimes_k R)\subset I\otimes_kR$ is a derivation-automorphism of $A(m;\bm{n})\otimes_k R$ if and only if $\varphi$ is an automorphism of the divided power ring $(A(m;\bm{n})\otimes_k R,I\otimes_kR)$. 
\end{lem}

\begin{proof}
Suppose that $\varphi$ is an automorphism of the divided power ring $(A(m;\bm{n})\otimes_k R,I\otimes_kR)$. Let $D\in W(m;\bm{n})\otimes_kR$ be an arbitrary element. We will show that $\varphi\circ D\circ\varphi^{-1}$ again a special derivation of $A(m;\bm{n})\otimes_k R$. For any $g\in I\otimes_k R$ and any $r\ge 1$, we have
\begin{align*}
(\varphi\circ D\circ\varphi^{-1})(g^{(r)})
&=\varphi(D(\varphi^{-1}(g^{(r)})))=\varphi(D(\varphi^{-1}(g)^{(r)}))\\
&=\varphi(\varphi^{-1}(g)^{(r-1)}D(\varphi^{-1}(g)))=g^{(r-1)}(\varphi\circ D\circ \varphi^{-1})(g).
\end{align*}
This proves that $\varphi\circ D\circ\varphi^{-1}\in W(m;\bm{n})\otimes_kR$. 

Conversely, suppose that $\varphi$ defines a derivation-automorphism of $A(m;\bm{n})\otimes_kR$. For any $f\in I\otimes_kR$ and any $r\ge 1$, we will show that $\varphi(f^{(r)})=\varphi(f)^{(r)}$ by induction on $r$. This is clear for $r=1$. Suppose that $r>1$. By the induction hypothesis, we have $\varphi(f^{(r-1)})=\varphi(f)^{(r-1)}$.  If $D\in W(m;\bm{n})\otimes_kR$, we have
\begin{align*}
D(\varphi(f^{(r)}))&=\varphi(\varphi^{-1}(D(\varphi(f^{(r)}))))
=\varphi((\varphi^{-1}\circ D\circ\varphi)(f^{(r)}))\\
&=\varphi(f^{(r-1)}(\varphi^{-1}\circ D\circ\varphi)(f))
=\varphi(f^{(r-1)})D(\varphi(f))\\
&=\varphi(f)^{(r-1)}D(\varphi(f))=D(\varphi(f)^{(r)}).
\end{align*}
This implies that $\varphi(f^{(r)})-\varphi(f)^{(r)}\in1\otimes_k R$. However, as $\varphi(f^{(r)})-\varphi(f)^{(r)}\in I\otimes_k R$, we have $\varphi(f^{(r)})-\varphi(f)^{(r)}=0$. This completes the proof.\end{proof}

As a consequence, any derivation-automorphism $\varphi$ with $\varphi(I\otimes_kR)\subset I\otimes_kR$ is completely determined by the images $\varphi(x_1),\dots,\varphi(x_m)$ of the variables $x_1,\dots,x_m$. 

\begin{definition}\label{def:G^0}
We define the subgroup scheme $G^0(m;\bm{n})$ (or $G^0$ shortly) of $G(m;\bm{n})$ by attaching any $k$-algebra $R$ to the subgroup $G^0(m;\bm{n})(R)\subset G(m;\bm{n})(R)$ consisting of $R$-automorphisms $g$ of the form $g=\varphi_*$, where $\varphi$ is an $R$-automorphism of the divided power algebra $(A(m;\bm{n})\otimes_k R,I\otimes_kR)$ such that 
\[
\varphi(x_{j})=\sum_{i=1}^ma_{ij}x_i\quad (1\le j\le m)
\]
for some $(a_{ij})\in\GL_m(R)$. By definition, there exists a natural isomorphism of $k$-group schemes $G^0(m;\bm{n})\simeq\GL_m$. 
Moreover, for any $k$-algebra $R$ and any $R$-automorphism $\varphi$ of the divided power algebra $(A(m;\bm{n})\otimes_k R,I\otimes_kR)$, we define the $R$-automorphism $\varphi^0$ to be 
\begin{equation}\label{eq:def:G^0}
\varphi^0(x_{j})=\sum_{i=1}^m\partial_i(\varphi(x_j))\,x_i\quad (1\le j\le m).
\end{equation}
\end{definition}

\begin{definition}\label{def:G^+}
We define the subgroup scheme $G^+(m;\bm{n})$ (or $G^+$ shortly) of $G(m;\bm{n})$ by attaching any $k$-algebra $R$ to the subgroup $G^+(m;\bm{n})(R)\subset G(m;\bm{n})(R)$ consisting of $R$-automorphisms $g$ of the form $g=\varphi_*$, where $\varphi$ is an $R$-automorphism of the divided power algebra $(A(m;\bm{n})\otimes_k R,I\otimes_kR)$ such that $\varphi^0={\rm id}$ (cf.\ (\ref{eq:def:G^0})).
\end{definition}

\begin{prop}\label{prop:G^- uni}
The subgroup scheme $G^+=G^{+}(m;\bm{n})$ is a smooth  unipotent $k$-group scheme. 
\end{prop}

\begin{proof}
Let $R$ be a $k$-algebra. If $\varphi_*\in G^+(R)$, then we have 
\[
\varphi(x_i)=x_i+\sum_{|\alpha|\ge 2} f_{i,\alpha}x^{(\alpha)}
\] 
for all $1\le i\le m$, where the $f_{i,\alpha}\in R$ are arbitrary. This implies that $G^+$ is isomorphic to an affine space $\A_k^{N}$ for some $N>0$. In particular, $G^+$ is a smooth over $k$. Let us show that $G^+$ is a unipotent $k$-group scheme. It suffices to show that any element $\varphi_*\in G^+(k)$ is unipotent, i.e.\,$\varphi-{\rm id}$ is nilpotent in ${\rm End}_k(A(m;\bm{n}))$ (cf.\cite[Chapter 8]{wat79}). As
$(\varphi-{\rm id})(x_i)=\sum_{|\alpha|\ge 2} f_{i,\alpha}x^{(\alpha)}$ 
with $f_{i,\alpha}\in k$ for any $1\le i\le m$, we have $(\varphi-{\rm id})^r(x_i)=\sum_{|\alpha|\ge 1+r}g_{i,r,\alpha}x^{(\alpha)}$ with $g_{i,r,\alpha}\in k$ for any $1\le i\le m$, which implies that $(\varphi-{\rm id})^l=0$ for sufficiently large $l\gg 1$. This completes the proof. 
\end{proof}

\subsection{Triangulation of the automorphism group schemes}\label{sec:triangulation}

We will use the same notation as in the previous subsection. Now we will show that the automorphism group schemes $G(m;\bm{n})$ admit \textit{triangulations} in the sense of Severitt\cite[Definition 3.4]{severitt}.

\begin{thm}\label{thm:triangulation}
For any $m\in\Z_{>0}$ and any $\bm{n}\in\Z_{>0}^m$, the multiplication map
\[
G^+\times G^0\times G^-\longrightarrow~G(m;\bm{n})
\]
is an isomorphism of $k$-schemes. 
\end{thm} 

\begin{proof}
Let $R$ be a $k$-algebra. We have to show that the multiplication map
\[
G^+(R)\times G^0(R)\times G^-(R)\longrightarrow~G(m;\bm{n})(R)
\]
is bijective. First we will show that this map is surjective. Let $\varphi$ be a derivation-automorphism of $A(m;\bm{n})\otimes_k R$. We define $\bm{a}=(a_{is})\in\prod_{i=1}^mW_{n_i(1)}(R)$ to be
\[ 
a_{is}\Def\varepsilon(\varphi(x_i^{(p^s)}))\quad (1\le i\le m,~0\le s\le n_i-1)
\]
and define $\varphi^-\Def\varphi(\bm{a})$ (cf.\,Definition \ref{def:G^-}). For the definition of the map $\varepsilon$, see (\ref{eq:aug}). We claim that the composition of the derivation-automorphisms $\varphi^{\ge 0}\Def\varphi\circ(\varphi^-)^{-1}=\varphi\circ\varphi(-\bm{a})$ (cf.\,Remark \ref{rem:AH p-torsion}) satisfies
\[
\varphi^{\ge 0}(I\otimes_kR)\subset I\otimes_k R.
\]
As explained in \cite{severitt}, $\varphi$ can be written as the composition $\varphi=\psi^{\ge 0}\circ\psi^{-}$ of automorphisms $\psi^{\ge 0}$, $\psi^{-}$ of the truncated polynomial ring $A(m;\bm{n})\otimes_kR$ so that $\psi^{\ge 0}(I\otimes_k R)\subset I\otimes_k R$ and $\psi^{-}$ is defined by the parallel transformation 
\[
\psi^{-}(x_i^{(p^s)})=a_{is}+x_i^{(p^s)}\quad (1\le i\le m,~0\le s\le n_i-1).
\] 
This reduces the problem to proving that $\psi^-\circ\varphi(-\bm{a})(I\otimes_k R)\subset I\otimes_k R$. By (\ref{eq:E_p x_i}), we have
\begin{align*}
\varphi(-\bm{a})(x_i^{(p^s)})
&=-a_{is}+\sum_{j_0,\dots,j_{s-1}=0}^{p-1}\left\{\prod_{t=0}^{s-1}(-a_{it})^{(j_t)}\right\}x_i^{(p^s-j_{s-1}p^{s-1}-\dots-j_1p-j_0)}\\
&=-a_{is}+x_i^{(p^s)}+\sum_{j_0=1}^{p-1}\sum_{j_1,\dots,j_{s-1}=0}^{p-1}\left\{\prod_{t=0}^{s-1}(-a_{it})^{(j_t)}\right\}x_i^{(p^s-j_{s-1}p^{s-1}-\dots-j_1p-j_0)}\\
&\quad+\sum_{j_1=1}^{p-1}\sum_{j_2,\dots,j_{s-1}=0}^{p-1}\left\{\prod_{t=1}^{s-1}(-a_{it})^{(j_t)}\right\}x_i^{(p^s-j_{s-1}p^{s-1}-\dots-j_1p)}\\
&\quad+\cdots\\
&\quad+\sum_{j_{s-2}=1}^{p-1}\sum_{j_{s-1}=0}^{p-1}\left\{\prod_{t=s-2}^{s-1}(-a_{it})^{(j_t)}\right\}x_i^{(p^s-j_{s-1}p^{s-1}-j_{s-2}p^{s-2})}
+\sum_{j_{s-1}=1}^{p-1}(-a_{is})^{(j_s)}x_i^{(p^s-j_{s-1}p^{s-1})}.
\end{align*}
As $\dfrac{p^s-j_{s-1}p^{s-1}-\cdots-j_{u}p^{u}}{p^u}\equiv p-j_t\mod p$ for $0\le u\le s-1$, we have 
\[
(x_i^{(p^u)})^{(p-j_u)}\mid x_i^{(p^s-j_{s-1}p^{s-1}-\cdots-j_{u}p^{u})}\quad (1\le j_u\le p-1),
\]  
hence 
\[
(a_{iu}+x_i^{(p^u)})^{(p-j_u)}\mid \psi^-(x_i^{(p^s-j_{s-1}p^{s-1}-\cdots-j_{u}p^{u})})\quad (1\le j_u\le p-1).
\]  
By the condition that $a_{i0}^p=a_{i1}^p=\cdots=a_{i,s-1}^p=0$, this proves that
\[
\psi^{-}((-a_{iu})^{(j_u)}x_i^{(p^s-j_{s-1}p^{s-1}-\cdots-j_{u}p^{u})})\in (-a_{iu})^{(j_u)}(a_{iu}+x_i^{(p^u)})^{(p-j_u)}A(m;\bm{n})\otimes_kR\subset I\otimes_kR
\] 
for any $1\le j_u\le p-1$. Therefore, we can conclude that $\psi^{-}(\varphi(-\bm{a})(I\otimes_k R))\subset I\otimes_k R$. 

We have seen that any derivation-automorphism $\varphi$ of $A(m;\bm{n})\otimes_kR$ can be decomposed into the composition $\varphi=\varphi^{\ge 0}\circ\varphi^-$ of derivation-automorphisms with $\varphi^{\ge 0}(I\otimes_k R)\subset I\otimes_kR$ and $\varphi^-\in G^-(R)$. If we define $\varphi^0\in G^0(R)$ by (\ref{eq:def:G^0}), then we have $\varphi^+\Def\varphi^{\ge 0}\circ(\varphi^0)^{-1}\in G^+(R)$. This proves that $\varphi_*$ is in the image of the multiplication map $G^+(R)\times G^0(R)\times G^-(R)\to G(m;\bm{n})(R)$.   

It remains for us to show that the map $G^+(R)\times G^0(R)\times G^-(R)\to G(m;\bm{n})(R)$ is injective. Suppose that we have two decompositions of a derivation-automorphism $\varphi$, namely $\varphi=\varphi^+\circ\varphi^0\circ\varphi^-=\psi^+\circ\psi^0\circ\psi^-$ with $\varphi^{\bullet},\psi^{\bullet}\in G^{\bullet}(R)$ for $\bullet\in\{+,0,-\}$. Let us take $\bm{a}=(a_{is}),\bm{b}=(b_{is})\in\prod_{i=1}^{m}W_{n_i(1)}(R)$ so that $\varphi^-=\varphi(\bm{a})$, $\psi^-=\varphi(\bm{b})$. Then we have 
\[
a_{is}=\varepsilon(\varphi^{-}(x_i^{(p^s)}))=\varepsilon(\varphi(x_i^{(p^s)}))=\varepsilon(\psi^-(x_i^{(p^s)}))=b_{is}\quad (1\le i\le m,0\le s\le n_i-1).
\]
This proves that $\varphi^-=\psi^-$. It remains to show that $\varphi^{\bullet}=\psi^{\bullet}$ for $\bullet\in\{+,0\}$. However, as 
\[
(\varphi\circ(\varphi^{-})^{-1})(x_i)\equiv\varphi^0(x_i)\equiv\psi^0(x_i)\mod \sum_{|\alpha|\ge 2}Rx^{(\alpha)}\quad (1\le i\le m),
\]
we have $\varphi^0=\psi^0$ and $\varphi^+=\varphi\circ(\varphi^-)^{-1}\circ(\varphi^0)^{-1}=\varphi\circ (\psi^-)^{-1}\circ (\psi^0)^{-1}=\psi^+$. This completes the proof. 
\end{proof}

\begin{rem}
In particular, for any reduced $k$-algebra $R$, we have 
\[
\Aut_R(W(m;\bm{n})\otimes_kR)=G(m;\bm{n})(R)=G^+(R)\rtimes G^0(R).
\] 
This should be compared with \cite[Theorem 12.13]{Ree} and \cite[Theorem 2]{Wilson}. 
\end{rem}

\begin{cor}\label{cor:triangulation}
Suppose that $p>3$. The classifying space $BG(m;\bm{n})$ is stably birationally equivalent to the classifying space $B\Gamma(m;\bm{n})$.   
\end{cor}

\begin{proof}
By Theorem \ref{thm:triangulation}, there exists an exact sequence of affine group schemes of finite type over $k$.
\begin{equation}\label{eq:cor triangulation}
1\lto\Gamma(m;\bm{n})\lto G(m;\bm{n})\xrightarrow{~F~} (G^+\rtimes G^0)^{(p)}\lto 1.
\end{equation}
As $(G^+)^{(p)}$ and $(G^0)^{(p)}$ are special algebraic groups (cf.\,Definition \ref{def:G^0} and Proposition \ref{prop:G^- uni}), so is the semi-direct product $(G^+\rtimes G^0)^{(p)}=(G^+)^{(p)}\rtimes (G^0)^{(p)}$ (cf.\,Remark \ref{rem:special}(2)). Therefore, the assertion follows from Proposition \ref{prop:special G''}.  
\end{proof}

The following consequence implies that purely inseparable extensions of exponent one are enough to trivialize any twisted forms of the generalized Witt algebras. 

\begin{cor}\label{cor2:triangulation}
Let $R$ be a regular local ring over $k$ with fraction field $K$. Let $L$ be a Lie algebra over $R$ which is a twisted form of $W(m;\bm{n})\otimes_kR$. If $R^{1/p}$ is the integral closure of $R$ in the maximal purely inseparable extension $K^{1/p}$ of $K$ of exponent one, then there exists an $R^{1/p}$-isomorphism of Lie algebras $L\otimes_RR^{1/p}\simeq W(m;\bm{n})\otimes_kR^{1/p}$. 
\end{cor}

\begin{proof}
As the isomorphism classes of the twisted forms of $W(m;\bm{n})\otimes_kR$ are exactly classified by the pointed set of cohomology classes $H^1_{\fppf}(R,G(m;\bm{n}))$. Therefore, it suffices to show that the restriction map $H^1_{\fppf}(R,G(m;\bm{n}))\to H^1_{\fppf}(R^{1/p},G(m;\bm{n}))$ is the trivial map. By the short exact sequence (\ref{eq:cor triangulation}), we have an exact sequence of pointed sets
\[
H^1_{\fppf}(R,\Gamma(m;\bm{n}))\lto H^1_{\fppf}(R,G(m;\bm{n}))\lto 1.
\] 
Notice here that $H^1_{\fppf}(R,G^+\rtimes G^0)=1$ again by the speciality of $G^+\rtimes G^0$. As $\Gamma(m;\bm{n})$ is a height one group scheme, this implies the triviality of the map $H^1_{\fppf}(R,G(m;\bm{n}))\to H^1_{\fppf}(R^{1/p},G(m;\bm{n}))$. This completes the proof. 
\end{proof}

Moreover, we obtain the following injectivity result for $G(m;\bm{n})$-torsors. 

\begin{cor}\label{cor3:triangulation}
Let $R$ be a regular local ring over $k$ with fraction field $K$. The kernel of the restriction map $H^1_{\fppf}(R,G(m;\bm{n}))\to H^1_{\fppf}(K,G(m;\bm{n}))$ is trivial. 
\end{cor}

\begin{proof}
Let $P\to\Spec R$ be an fppf $G(m;\bm{n})$-torsor such that the base change $P\otimes_RK\to\Spec K$ is a trivial $G(m;\bm{n})$-torsor. We have to show that $P(R)\neq\emptyset$. Let $Q\Def P/\Gamma(m;\bm{n})\to\Spec R$ be the induced $H\Def G(m;\bm{n})/\Gamma(m;\bm{n})$-torsor. The short exact sequence (\ref{eq:cor triangulation}) implies that $H$ is special (see also the proof of Corollary \ref{cor:triangulation}). As $R$ is a local ring, the $H$-torsor $Q\to\Spec R$ admits a section $s\colon\Spec R\to Q$ (cf.\ Definition \ref{def:special}), which defines an isomorphism of $R$-schemes $H\otimes R\xrightarrow{\simeq}Q$. Via this isomorphism, we consider $P\to Q$ as a $\Gamma(m;\bm{n})$-torsor over $H\otimes_kR$. If we write $R=\varinjlim_i R_i$ as a colimit of smooth $k$-algebras $R_i$, the $\Gamma(m;\bm{n})$-torsor $P\to H\otimes_k R$ extends to a $\Gamma(m;\bm{n})$-torsor $\widetilde{P}\to H\otimes_k R_i$ for some $i$ with $\widetilde{P}\otimes_{R_i}{\Frac R_i}\simeq G(m;\bm{n})\otimes_k\Frac R_i$. This implies that there exist an affine dense open subset $U\subset \Spec R_i$ and an isomorphism of $\Gamma(m;\bm{n})$-torsors $\widetilde{P}|_U\simeq G(m;\bm{n})\times U$ over $H\times_kU$. As $\Gamma(m;\bm{n})$ is finite $k$-group scheme, the surjectivity of the restriction homomorphism between Nori's fundamental group schemes $\pi^{\rm N}(H\times_kU)\twoheadrightarrow\pi^{\rm N}(H\times_k\Spec R_i)$ \cite[Chapter II, \S2]{Nori}\cite[Theorem II]{RTZ}  implies that there exists an isomorphism of $\Gamma(m;\bm{n})$-torsors $\widetilde{P}\simeq G(m;\bm{n})\times_k\Spec R_i$ over $H\times_k\Spec R_i$. Therefore, we have an isomorphism of schemes $P\simeq G(m;\bm{n})\otimes_kR$ over $H\otimes_kR$ and they are $R$-isomorphic to each other. This implies that $P(R)\neq\emptyset$ as desired. This completes the proof.  
\end{proof}


\subsection{The Witt--Ree algebras over general base rings}\label{sec:Witt-Ree}

In this subsection, we introduce a relative version of the notion of \textit{Witt--Ree algebra} in the sense of \cite{Waterhouse} and discuss their properties.  
Let $k$ be a field of characteristic $p>0$ and $R$ a $k$-algebra. An $R$-algebra $A$ is said to be {\it purely inseparable of height one} if it is isomorphic to an $R$-algebra of the form $R[x_1,\dots,x_n]/(x_1^p-a_1,\dots,x_n^p-a_n)$ for some $n$ and some $a_i\in R$. 
Let $A=R[x_1,\dots,x_n]/(x_1^p-a_1,\dots,x_n^p-a_n)$ be a purely inseparable $R$-algebra of height one. Then the derivation algebra $\Der_R(A)$ on $A$ over $R$ is a free $A$-module with basis $\delta_1,\dots,\delta_n$, where $\delta_i(x_j)=\delta_{ij}$ is the Kronecker delta.

\begin{definition}\label{def:Witt-Ree}
Let $A$ be a purely inseparable algebra of height one over $R$. 
A Lie subalgebra $L\subset\Der_R(A)$ is called a \textit{Witt--Ree $R$-algebra} on $A$ if the following conditions are satisfied.
\begin{itemize}
\item[(WR1)] $L$ is a finite free $A$-submodule of $\Der_R(A)$. 
 
\item[(WR2)] Only the constants $R\cdot 1$ inside $A$ are annihilated by all $D\in L$.

\item[(WR3)] 
The center $Z(L)$ of $L$ is trivial, and for any maximal ideal $\frkm\subset R$ of $R$ with $\kappa(\frkm)=R_{\frkm}/\frkm R_{\frkm}$, the Lie algebra $L\otimes_R\kappa(\frkm)$ is a central simple Lie algebra over $\kappa(\frkm)$ in the sense of \cite[Chapter X,\S1]{Jacobson}. 
\end{itemize}
Moreover, $L$ is said to be of \textit{$D$-dimension $m$} if $L$ has  rank $m$ over $A$. 
\end{definition}

\begin{rem}
In the case where $R=\kappa$ is a field over $k$, then the above definition of Witt--Ree algebra is the same as the definition due to Waterhouse \cite[Section 1]{Waterhouse}.
\end{rem}

\begin{rem}\label{rem:Der A base change}
For a purely inseparable $R$-algebra of height one $A=R[x_1,\dots,x_n]/(x_1^p-a_1,\dots,x_n^p-a_n)$, we have $\Der_R(A)=\sum_{i=1}^nA\delta_i\simeq A^{\oplus n}$, where $\delta_{i}(x_j)=\delta_{ij}$. For any $k$-algebra homomorphism $R\to S$, we have $\Der_R(A)\otimes_RS=\Der_S(A\otimes_RS)$ and the natural restriction map of Lie algebras $\Der_R(A)\to\Der_S(A\otimes_RS)=\Der_R(A)\otimes_RS$ is nothing other than the base extension map $D\mapsto D\otimes {\rm id}_S$. In particular, if $L\subset\Der_R(A)$ is a Lie subalgebra over $R$, then $L\otimes_RS\subset\Der_S(A\otimes_RS)$ is a Lie subalgebra over $S$. 
\end{rem}

\begin{rem}\label{rem:orthogonal system}
Let $(R,\frkm)$ be a local integral domain over $k$ with $\kappa\Def R/\frkm$ residue field. 
Let $L\subset\Der_R(A)$ be a Witt--Ree $R$-algebra on $A$ of $D$-dimension $m$. Then $L$ admits a \textit{orthonormal system $\bfD=
\{D_1,\dots,D_m\}$ of derivations} (cf.\cite[\S3]{Ree}), i.e.\ there exist derivations $D_1,\dots,D_m\in\Der_R(A)$ and elements $g_1,\dots,g_m\in A$ such that  $D_i(g_j)=\delta_{ij}$ for any $i,j$, and moreover we have $L=\sum_{i=1}^mA \,D_i$. Indeed, let $\mathfrak{M}\subset A$ be the inverse image of the maximal ideal of $A=A\otimes_R\kappa=A/\frkm A$. Then $A$ is a local algebra with maximal ideal $\mathfrak{M}$. Let $\overline{L}\Def L\otimes_R\kappa\subset\Der_{\kappa}(A\otimes\kappa)$ be the base change of $L$ along the reduction map $R\twoheadrightarrow \kappa$. Let $\bfD=\{D_1,\dots, D_m\}$ be a free $A$-basis of $L$. Let $\overline{D}_i\Def D_i\otimes 1\in\overline{L}$ and $\overline{\bfD}\Def \{\overline{D}_1,\dots,\overline{D}_m\}$.  As $A\otimes\kappa$ is completely primary, by \cite[Theorem 3.5]{Ree}, there exist $f_1,\dots,f_m\in A\otimes\kappa$ such that $\Det (\overline{D}_i(f_j))\in (A\otimes\kappa)^*$. If $g_j\in A$ is a lift of $f_j$ for each $j$, then we have the image of $D_i(g_j)\in A$ in $A\otimes\kappa$ coincides with $\overline{D}_i(f_j)$ for any $i,j$. Therefore, $\Det (D_i(g_j))\in A^*=A\setminus \mathfrak{M}$. Let $(c_{ij})$ be the inverse matrix of the matrix $(D_i(g_j))$ and we set $D_i'\Def\sum_{j}c_{ij}D_j$. Then the elements $g_1,\dots,g_m$ satisfy $D'_i(g_j)=\delta_{ij}$ for any $i,j$. Therefore, $\bfD=\{D_1,\dots, D_m\}$ is equivalent to the orthonormal system of derivations $\bfD'\Def \{D_1',\dots,D'_m\}$. 
\end{rem}

\begin{lem}\label{lem:WR central simple}
Let $R$ be an integral domain over $k$ with fraction field $K$ and $L\subset\Der_R(A)$ a Witt--Ree $R$-algebra on $A$. 
\begin{enumerate}
\item $\End_{R\text{-lin}}(L)$ is generated as an $R$-algebra by $[D,-]$ and $[-,D]$ for all $D$ in $L$. 

\item $L\otimes_RK$ is central simple over $K$. 
\end{enumerate}
\end{lem}

\begin{proof}
\begin{enumerate}
\item Let $L^{e}\subseteq\End_{R\text{-lin}}(L)$ be the $R$-subalgebra generated by $[D,-]$ and $[-,D]$ for all $D$ in an $R$-basis of $L$. It suffices to show that  $L^e\otimes R_{\frkm}=\End_{R\text{-lin}}(L)\otimes R_{\frkm}$ for any maximal ideal $\frkm$ of $R$. Let $\frkm$ be an arbitrary maximal ideal of $R$. Note that $L^e\otimes R_{\frkm}=(L\otimes R_{\frkm})^e\subseteq\End_{R_{\frkm}\text{-lin}}(L\otimes R_{\frkm})=\End_{R\text{-lin}}(L)\otimes R_{\frkm}$. On the other hand, by the condition (WR3), we have $L^e\otimes\kappa(\frkm)=(L\otimes\kappa(\frkm))^e=\End_{\kappa(\frkm)\text{-lin}}(L\otimes\kappa(\frkm))$ (cf.\cite[p.293]{Jacobson}\cite[Lemma 1.1]{Waterhouse}). Nakayama's lemma thus implies that $L^e\otimes R_{\frkm}=\End_{R_{\frkm}\text{-lin}}(L\otimes R_{\frkm})$. This completes the proof.

\item Let $Z(L\otimes_RK)$ be the center of $L\otimes_RK$. For any $D\in Z(L\otimes_RK)$, there exists an element $a\in R$ such that $aD\in L\cap Z(L\otimes_RK)\subseteq Z(L)=0$, hence $D=0$. This implies that $Z(L\otimes_RK)=0$. It remains to prove that $(L\otimes_RK)^e=\End_{K\text{-lin}}(L\otimes_RK)$ (cf.\cite[p.293]{Jacobson}\cite[Lemma 1.1]{Waterhouse}). However, by (1), we have $L^e=\End_{R\text{-lin}}(L)$, hence $(L\otimes_RK)^e=L^e\otimes_RK=\End_{R\text{-lin}}(L)\otimes_RK=\End_{K\text{-lin}}(L\otimes_RK)$. This completes the proof. 
\end{enumerate}
\end{proof}

\begin{rem}\label{rem:WR central simple}
\begin{enumerate}
\renewcommand{\labelenumi}{(\arabic{enumi})}
\item Let $L\subset\Der_R(A)$ be an arbitrary Lie $R$-subalgebra with $L$ center free. Let $L^e\subseteq\End_{R\text{-lin}}(L)$ be the $R$-subalgebra generated by $[D,-]$ and $[-,D]$ for all $D\in L$. Then the argument in the proof of the previous lemma implies that the following conditions are equivalent to each other.
\begin{enumerate}
\renewcommand{\labelenumii}{(\roman{enumii})}
\item $L^e=\End_{R\text{-lin}}(L)$. 

\item For any maximal ideal $\frkm\subset R$ of $R$ with $\kappa(\frkm)=R_{\frkm}/\frkm R_{\frkm}$, the Lie algebra $L\otimes_R\kappa(\frkm)$ is a central simple Lie algebra over $\kappa(\frkm)$. 
\end{enumerate}

\item The algebra $L^e$ is called the \textit{enveloping $R$-algebra} for $L$. For a center free Lie $R$-algebra, condition (WR3) is equivalent also to saying that the enveloping algebra $L^e$ is an \textit{ideal $R$-algebra} in the sense of Ranga Rao\cite{Rao}. 
\end{enumerate}
\end{rem}

\begin{prop}\label{prop:WR base change}
Let $f\colon R\to S$ be a faithfully flat $k$-algebra homomorphism of integral domains over $k$. The base change $L\otimes_R S\subset\Der_{R}(A)\otimes_RS=\Der_{S}(A\otimes_RS)$ of a Witt--Ree $R$-algebra $L\subset\Der_R(A)$ defines a Witt--Ree $S$-algebra on $A\otimes_RS$ of the same $D$-dimension. 
\end{prop}

\begin{proof}
If $\bfD=\{D_1,\dots,D_m\}$ is a free $A$-basis of $L$, then $\bfD\otimes S=\{D_1\otimes{\rm id}_S,\dots,D_m\otimes{\rm id}_S\}$ is a free $A\otimes_RS$-basis of $L\otimes_RS$. Hence, the condition (WR1) is satisfied for $L\otimes_RS$. 
We will check the condition (WR2) for $L\otimes_RS$. Let $\{u_1,u_2,\dots,u_N\}$ be a free $R$-basis of $A$ with $u_1=1$. Suppose that $f=\sum_{i=1}^{N}\alpha_i u_i\in A\otimes_RS$ with $\alpha_i\in S$ satisfies the condition that $(D_i\otimes{\rm id}_S)(f)=0$ for all $1\le i\le m$, i.e.\, 
\[
\alpha_2D_i(u_2)+\alpha_3D_i(u_3)+\cdots+\alpha_ND_i(u_N)=0\quad (1\le i\le m).
\]
If $\alpha_2,\alpha_3,\dots,\alpha_N\in S\subset\Frac S$ were not all zero, by the argument in the proof of \cite[Lemma 7.8]{Ree}, there exist $\beta_2,\beta_3,\dots,\beta_N\in K=\Frac R$, which are not all zero, such that 
\[
\beta_2D_i(u_2)+\beta_3D_i(u_3)+\cdots+\beta_ND_i(u_N)=0\quad (1\le i\le m).
\]
By multiplying some element of $R$, this implies that there exist $\gamma_2,\gamma_3,\dots,\gamma_N\in R$, which are not all zero, so that 
\[
\gamma_2D_i(u_2)+\gamma_3D_i(u_3)+\cdots+\gamma_ND_i(u_N)=0\quad (1\le i\le m).
\]
This implies that $g\Def\sum_{i=2}^N\gamma_iu_i\in A\setminus R$ satisfies $D_i(g)=0$ for $1\le i\le m$. This contradicts the condition (WR2) for $L$. 
Therefore, we must have $\alpha_2=\alpha_3=\cdots=\alpha_N=0$ and $f=\alpha_1\in S$, hence the condition (WR2) is satisfied for $L\otimes_R S$. 

Finally, we will check the condition (WR3) for $L\otimes_RS$. As the base change of the adjoint map ${ad}\otimes {\rm id}_{\Frac S}\colon L\otimes\Frac S\to\End_{R\text{-lin}}(L)\otimes_R\Frac S=\End_{\Frac S\text{-lin}}(L\otimes_R\Frac S)$ is injective, so is the adjoint  map ${ad}\colon L\otimes_RS\to\End_{S\text{-lin}}(L\otimes_RS)$. This implies that $L\otimes_RS$ is center free. Furthermore, we have $(L\otimes_RS)^{e}=L^e\otimes_RS=\End_{R\text{-lin}}(L)\otimes_RS=\End_{S\text{-lin}}(L\otimes_RS)$. Therefore, the condition (WR3) is fulfilled for $L\otimes_RS$ (cf.\,Remark \ref{rem:WR central simple}(1)). 
\end{proof}

\begin{rem}\label{rem:WR base change}
In particular, if $R$ is an integral domain with fraction field $K=\Frac R$ and if $L$ is a Witt--Ree $R$-algebra on $A$, for any prime ideal $\frkp$ of $R$, the localization $L\otimes_RR_{\frkp}$ is a Witt--Ree $R_{\frkp}$-algebra on $A\otimes_RR_{\frkp}$ of the same $D$-dimension. 
\end{rem}

\begin{prop}\label{prop:WR descent}
Let $(R,\frkm)$ be a local integral domain over $k$ and $A$ a purely inseparable $R$-algebra of height one. 
Let $f\colon R\to S$ be a faithfully flat $k$-algebra homomorphism. 
Let $L\subset\Der_R(A)$ be a Lie $R$-subalgebra of $R$-linear derivations on $A$.  
If the base change $L\otimes_R S\subset\Der_{R}(A)\otimes_RS=\Der_{S}(A\otimes_RS)$ is a Witt--Ree $S$-algebra on $A\otimes_RS$ of $D$-dimension $m$, 
then $L$ is a Witt--Ree $R$-algebra on $A$ of $D$-dimension $m$.
\end{prop}

\begin{proof}
By assumption, the Lie subalgebra $L\otimes_RS\subset\Der_S(A\otimes_RS)$ satisfies the conditions (WR1), (WR2) and (WR3). In particular, $L\otimes_RS$ is a free module of rank $m$ over $A\otimes_RS$. As the map $A\to A\otimes_RS$ is faithfully flat, this implies that $L$ is a projective $A$-module. However, as $R$ is local, so is the purely inseparable $R$-algebra of height one $A$ (cf.\ \cite[Proof of Lemma 3.4]{Waterhouse}). Therefore, $L$ is a free $A$-module of rank $m$, hence (WR1) is satisfied for $L$. Let $\bfD=\{D_1,\dots,D_m\}$ be a free $A$-basis of $L$. 
Let us show that $L$ satisfies the conditions (WR2) and (WR3). Let $\{u_1,u_2,\dots,u_N\}$ be a free $R$-basis of $A$ with $u_1=1$. Let $f=\sum_{i=1}^N\alpha_i u_i\in A$ with $\alpha_i\in R$ be an element which is annihilated by all $D\in L$. As $\bfD\otimes S=\{D_1\otimes{\rm id}_S,\dots,D_m\otimes{\rm id}_S\}$ is a free $A\otimes_RS$-basis of $L\otimes_RS$, the condition (WR2) for $L\otimes_RS$ implies that $\alpha_2=\cdots=\alpha_N=0$. Hence, we have $f=\alpha_1u_1=\alpha_1\in R$ and the condition (WR2) for $L$ holds true. 

It remains to verify the condition (WR3). As the adjoint map $ad\colon L\to\End_{R\text{-lin}}(L)$ is injective after the base change to $S$, it is injective over $R$, hence the center $Z(L)$ of $L$ is trivial. Moreover, let $\frkp\subset S$ be a prime ideal of $S$ lying above $\frkm$. Then by the condition (WR3) for $L\otimes_RS_{\frkp}$ (cf.\,Remark \ref{rem:WR base change}), we find that $L\otimes_R\kappa(\frkp)=L\otimes_{\kappa(\frkm)}\kappa(\frkp)$ is a central simple Lie algebra over $\kappa(\frkp)$. By the argument of the proof for \cite[Lemma 1.1]{Waterhouse}, this implies that $L$ is a central simple Lie algebra over $\kappa(\frkm)$, hence the condition (WR3) for $L$. This completes the proof.
\end{proof}

Before closing this subsection, we will extend Waterhouse's Theorem \ref{thm:Aut W} to more general base rings (see Proposition \ref{prop:Aut WR}). To this end, we show several lemmas. 

\begin{lem}\label{lem:WR reduction}
Let $R$ be an integral domain over $k$ with fraction field $K$ and $L\subset\Der_R(A)$ a Witt--Ree $R$-algebra on a purely inseparable $R$-algebra $A$ of height one. For any prime ideal $\frkp$ of $R$, the reduction $L\otimes\kappa(\frkp)\subset \Der_{\kappa(\frkp)}(A\otimes\kappa(\frkp))$ at $\frkp$ is a Witt--Ree $\kappa(\frkp)$-algebra on $A\otimes\kappa(\frkp)$. 
\end{lem}

\begin{proof}
By Lemma \ref{lem:WR central simple}(1), we have $L^e=\End_{R\text{-lin}}(L)$, which implies that 
\[
\left(L\otimes\overline{\kappa({\frkp})}\right)^e=\End_{\overline{\kappa(\frkp)}\text{-lin}}\left(L\otimes\overline{\kappa(\frkp)}\right).
\]  
Therefore, $L\otimes\overline{\kappa(\frkp)}$ is central simple over $\overline{\kappa(\frkp)}$ (cf.\cite[Chapter X,\S1]{Jacobson}). Moreover, if $C\subset A\otimes\overline{\kappa}(\frkp)$ denotes the subalgebra of constants for $L\otimes\overline{\kappa(\frkp)}$, then by \cite[Lemma 3.2]{Ree}, $C$ is an integral domain. Hence, $\overline{\kappa(\frkp)}\subseteq C\subseteq\left(A\otimes\overline{\kappa(\frkp)}\right)_{\rm red}=\overline{\kappa(\frkp)}$. This proves that only the constants $\overline{\kappa(\frkp)}\cdot 1$ insides $A\otimes\overline{\kappa(\frkp)}$ are annihilated by all $D\in L\otimes\overline{\kappa(\frkp)}$. This completes the proof. 
\end{proof}

\begin{lem}\label{lem:Aut WR fp}
Let $R$ be a $k$-algebra and $A$ a purely inseparable $R$-algebra of height one. Let $L\subset\Der_R(A)$ be a Witt--Ree $R$-algebra on $A$. Then the automorphism $R$-group scheme $\uAut_{R}(L)$ of the Lie algebra $L$ over $R$ admits a faithful representation $\rho\colon\uAut_R(L)\hookrightarrow\GL_{N,R}$ with some $N>0$, and $\uAut_R(L)$ is an affine $R$-group scheme locally of finite presentation. 
\end{lem}

\begin{proof}
As $L$ is a finite free module over $R$, this is immediate from \cite[Lemma 2.3]{GP}. 
\end{proof}

\begin{lem}\label{lem:Aut WR fppf}
Let $(R,\frkm)$ be a Noetherian local integral domain over $k$ with fraction field $K$ and residue field $\kappa$. Let $A$ be a purely inseparable $R$-algebra of height one. Let $L\subset\Der_R(A)$ be a Witt--Ree $R$-algebra on $A$ of $D$-dimension $m$. Suppose that there exists an $m$-tuple of positive integers $\bm{n}\in\Z_{>0}^m$ such that $L\otimes_R\overline{K}\simeq W(m;\bm{n})\otimes_k\overline{K}$ and $L\otimes_{R}\overline{\kappa}\simeq W(m;\bm{n})\otimes_k\overline{\kappa}$ as Lie algebras over algebraic closures. Then the automorphism $R$-group scheme $\uAut_{R}(L)$ is an affine flat $R$-group scheme of finite presentation. Similarly for the $R$-subgroup scheme $\uAut_R(A,L)$ of the automorphism group scheme $\uAut_R(A)$ consisting of automorphisms of $A$ which preserve the Lie subalgebra $L\subset\Der_R(L)$\end{lem}

\begin{proof}
We will prove the claim only for $\uAut_R(L)$, but the same argument proves the claim for $\uAut_R(A,L)$.   
To ease the notation, we put $G\Def \uAut_R(L)$. 
By the previous lemma, $G$ is an affine $R$-group scheme of finite presentation. Therefore, it suffices to show that $G$ is flat over $R$. As $G$ is of finite presentation, the relative Frobenius morphism $F\colon G\to G^{(p)}$ is a finite morphism and its kernel $G_{(1)}\Def\Ker (F)$ is a finite $R$-group scheme. 

We claim that $G_{(1)}$ is flat over $R$. Indeed, by the assumption that $L\otimes_R\overline{K}\simeq W(m;\bm{n})\otimes_k\overline{K}$, we have an isomorphism of $\overline{K}$-group schemes
\[
G_{(1)}\otimes_R\overline{K}= (G\otimes_R\overline{K})_{(1)}\simeq (G(m;\bm{n})\otimes_k\overline{K})_{(1)}\simeq\Gamma(m;\bm{n})\otimes_k\overline{K}.
\]
Here, recall that $\Gamma(m;\bm{n})=G(m;\bm{n})_{(1)}$ (cf.\ Section \ref{sec:autom}). Similarly, we have $G_{(1)}\otimes_R\overline{\kappa}\simeq\Gamma(m;\bm{n})\otimes_k\overline{\kappa}$. This implies that 
\[
\Dim_{K}R[G_{(1)}]\otimes_RK=\Dim_kk[\Gamma(m;\bm{n})]=\Dim_{\kappa}R[G_{(1)}]\otimes_R\kappa,
\]
where $R[G_{(1)}]=\Gamma(G_{(1)},\scrO_{G_{(1)}})$ is the coordinate ring of the finite $R$-group scheme $G_{(1)}$. This implies that the finite $R$-module $R[G_{(1)}]$ is projective. Hence, by \cite[Lemma 02KB]{stack}, $G_{(1)}$ is a finite flat $R$-group scheme. 

By \cite[Expos\'e V, Corollaire 10.1.3]{SGA3}, the fppf sheafification of the functor $A\mapsto G(A)/G_{(1)}(A)$ is representable by an $R$-group scheme $H$ of finite presentation and the map $F\colon G\to G^{(p)}$ factors as
\[
\xymatrix{
G\ar[rd]_{\pi}\ar[rr]^{F}&&G^{(p)}\\
&H\ar[ru]_{\iota}&
}
\]
where $\pi$ is the natural projection and $\iota$ is a monomorphism. Moreover, this factorization is stable under base change. By the flatness of $G_{(1)}$, the morphism $\pi\colon G\to H$ is flat. Therefore, to prove the flatness of $G$, we have only to show that $H$ is flat over $R$. Indeed, we can prove that $H$ is a smooth over $R$. By the exact sequence (\ref{eq:cor triangulation}) in the proof of Corollary \ref{cor:triangulation}, we have that both of the generic fiber $H\otimes_RK$ and the special fiber $H\otimes_R\kappa$ are twisted forms of the smooth $k$-group scheme $G(m;\bm{n})/\Gamma(m;\bm{n})=G^+\rtimes G^0$. Let $d\Def\Dim G^+\rtimes G^0$. Let $\mathfrak{p}\in\Spec R$ be an arbitrary point. By \cite[Expos\'e VI${}_B$, Proposition 4.1]{SGA3}, the subset
\[
S_{\mathfrak{p}}\Def\{\mathfrak{q}\in\Spec R\,|\,\Dim H\otimes\kappa(\mathfrak{q})\ge \Dim{H\otimes{\kappa(\mathfrak{p})}}\}
\]
is a closed subset of $\Spec R$. As $\mathfrak{p}\in S_{\mathfrak{p}}$, we have $\overline{\{\mathfrak{p}\}}\subseteq S_{\mathfrak{p}}$.We have $\mathfrak{m}\in\overline{\{\mathfrak{p}\}}\subseteq S_{\mathfrak{p}}$, hence $d=\Dim H\otimes\kappa\ge\Dim H\otimes\kappa(\mathfrak{p})$. Similarly, we have $\Dim H\otimes\kappa(\mathfrak{p})\ge \Dim H\otimes K=d$. Therefore, we have $\Dim H\otimes\kappa(\mathfrak{p})=d$. On the other hand, by the semicontinuity property for Lie algebras \cite[Lemma 2.9]{GP}, we can also find  that $\Dim_{\kappa(\mathfrak{p})} \Lie (H\otimes\kappa(\mathfrak{p}))=d\,(=\Dim H\otimes\kappa(\mathfrak{p}))$. This proves that $H\otimes\kappa(\mathfrak{p})$ is a smooth $\kappa(\mathfrak{p})$-group scheme. As $\mathfrak{p}\in\Spec R$ is arbitrary, by \cite[Expos\'e VI${}_B$, Corollaire 4.4]{SGA3}, we can conclude that $H$ is a smooth $R$-group scheme. This completes the proof. 
\end{proof}

Now we prove the next result.  

\begin{prop}\label{prop:Aut WR}
Suppose that $k=\overline{k}$ is an algebraically closed field of characteristic $p> 3$. 
Let $R$ be an integral domain over $k$ with fraction field $K$ and $L\subset\Der_R(A)$ a Witt--Ree $R$-algebra of $D$-dimension $m$ on a purely inseparable $R$-algebra $A$ of height one. We denote by $\uAut_R(A,L)$ the $R$-subgroup scheme of the automorphism group scheme $\uAut_R(A)$ consisting of automorphisms of $A$ which preserve the Lie subalgebra $L\subset\Der_R(L)$. Suppose that there exists an $m$-tuple of positive integers $\bm{n}\in\Z_{>0}^m$ such that $L\otimes_R\overline{K}\simeq W(m;\bm{n})\otimes_k\overline{K}$ and $L\otimes_{R}\overline{\kappa}\simeq W(m;\bm{n})\otimes_k\overline{\kappa}$ as Lie algebras over algebraic closures. 
The natural map 
\[
\uAut_R(A,L)\lto\uAut_R(L)~;~\varphi\mapsto \varphi_*=\varphi\circ(-)\circ\varphi^{-1}
\]
to the automorphism group scheme $\uAut_R(L)$ of the $R$-Lie algebra $L$ is an isomorphism of $R$-group schemes.  
\end{prop}

\begin{proof}
By the previous lemma, the group schemes in both sides are flat and of finite presentation over $R$. Therefore, by \cite[Expos\'e I, Proposition 5.7; Expos\'e VIII, Corollaire 5.4]{SGA1}, it suffices to check that for any prime ideal $\frkp$ of $R$, the induced homomorphism of group schemes
\[
\uAut_{\overline{\kappa(\frkp)}}\left(A\otimes\overline{\kappa(\frkp)},L\otimes\overline{\kappa(\frkp)}\right)\lto\uAut_{\overline{\kappa(\frkp)}}\left(L\otimes\overline{\kappa(\frkp)}\right)
\]
is an isomorphism over the algebraic closure $\overline{\kappa(\frkp)}$. By Lemma \ref{lem:WR reduction}, $L\otimes\overline{\kappa(\frkp)}$ is a Witt--Ree $\overline{\kappa(\frkp)}$-algebra on $A\otimes\overline{\kappa(\frkp)}$. Hence, the claim follows from Wilson's classification theorem for Witt--Ree algebras over  algebraically closed fields \cite[Theorem 1]{Wilson} and Waterhouse's Theorem \ref{thm:Aut W}.
\end{proof}

\begin{rem}
For the application to the rationality problem, we need to figure out which Lie algebra over the base ring $R$ is a twisted form of the generalized Witt algebra $W(m;\bm{n})\otimes_kR$ (cf.\ Proposition \ref{prop:ret rat BG(m,n)}). In the case where the base ring $R=K$ is a field, a complete classification of twisted forms of the generalized Witt algebra $W(m;\bm{n})\otimes_kK$ is achieved by the series of works due to Ree\cite{Ree}, Wilson\cite{Wilson} and Waterhouse \cite{Waterhouse}. They proved that a Lie algebra $L$ over $K$ is a twisted form of $W(m;\bm{n})\otimes_kK$ if and only if it is a Witt--Ree algebra over $K$ of the same type $(m;\bm{n})$. However, the author is not sure if this classification theory can be extended to an arbitrary base ring $R$. Proposition \ref{prop:Aut WR} shows that our Witt--Ree algebras over $R$ satisfy one of necessary conditions to be a twisted form of $W(m;\bm{n})$ under the mild assumption.  
\end{rem}


\subsection{Retract rationality for the generalized Witt algebras}\label{sec:RP Witt}

The goal of this subsection is to prove the next theorem.

\begin{thm}\label{thm:ret rat BGamma(m,n)}
Suppose that $k=\overline{k}$ is an algebraically closed field of characteristic $p>3$. If $\bm{n}=\bm{1}$ or $m=1$, 
the classifying space $B\Gamma(m;\bm{n})$ of the finite simple group scheme $\Gamma(m;\bm{n})$ is retract rational over $k$. 
\end{thm}

 As the retract rationality is stably birational invariant (cf.\ Remark \ref{rem:def ret rat}), thanks to Corollary \ref{cor:triangulation}, it suffices to prove the retract rationality for $BG(m;\bm{n})$. Moreover, by replacing the generically free representation $\rho\colon G(m;\bm{n})\to\GL_V$ with a higher dimensional generically free representation $\rho'\colon G(m;\bm{n})\to\GL_{V'}$, we may assume that the transcendental degree of $k(BG(m;\bm{n}))$ over $k$ is so large that the condition that 
\[
\mathrm{dim}_{\kappa}\,\Omega_{\kappa/k}^1\ge n_1+\cdots+n_m
\] 
in the next proposition is satisfied for $\kappa=k(BG(m;\bm{n}))$. 
By Proposition \ref{prop:ret rat}(f), the theorem is a consequence of the next result. 

\begin{prop}\label{prop:ret rat BG(m,n)}
Suppose that $k=\overline{k}$ is an algebraically closed field of characteristic $p>3$. Let $(R,\frkm)$ be the localization of a polynomial ring $k[X_1,\dots,X_N]$ at a prime ideal $\frkP\subset k[X_1,\dots,X_N]$ with residue field $\kappa\Def R/\frkm$ with $\Dim_{\kappa}\Omega^1_{\kappa/k}\ge n_1+\cdots+n_m$. If $\bm{n}=\bm{1}$ or $m=1$, the natural map 
\[
H^1_{\fppf}(R,G(m;\bm{n}))\lto H^1_{\fppf}(\kappa,G(m;\bm{n}))
\] 
is surjective. 
\end{prop}

\begin{proof}[Proof of Proposition \ref{prop:ret rat BG(m,n)} in the case where $\bm{n}=\bm{1}$] We first prove the proposition in the case when $\bm{n}=\bm{1}$. In this case, Theorem \ref{thm:Aut W} implies that there exists a natural isomorphism 
\[
\uAut(A(m;\bm{1}))=\uAut(A(m;\bm{1}),W(m;\bm{1}))\xrightarrow{~\simeq~}G(m;\bm{1})
\]
of $k$-group schemes (cf.\cite{Waterhouse0}). Therefore, it suffices to show that the natural map 
\[
H^1_{\fppf}(R,\uAut(A(m;\bm{1})))\to H^1_{\fppf}(\kappa,\uAut(A(m;\bm{1})))
\] 
is surjective.  
Each element $\xi\in H^1_{\fppf}(\kappa,\uAut(A(m;\bm{1})))$ is represented by a $\kappa$-form $A$ of $A(m;\bm{1})$, which is isomorphic to a purely inseparable $\kappa$-algebra of height one (cf.\cite{Waterhouse0}). There exist elements $a_1,\dots,a_m\in \kappa$ such that
\begin{equation*}
A\simeq \kappa[x_1,\dots,x_m]/(x_1^p-a_1,\dots,x_m^p-a_m). 
\end{equation*}
By taking lifts $\widetilde{a}_i\in R$ of $a_i\in \kappa$, we get an $R$-algebra
\[
\widetilde{A}\Def R[x_1,\dots,x_m]/(x_1^p-\widetilde{a}_1,\dots,x_m^p-\widetilde{a}_m),
\] 
which defines a class $\xi\Def [\widetilde{A}]\in H^1_{\fppf}(R,\uAut(A(m;\bm{1})))$ whose image in $H^1_{\fppf}(\kappa,\uAut(A(m;\bm{1})))$ coincides with $\xi$. This completes the proof. 
\end{proof}

\begin{proof}[Proof of Proposition \ref{prop:ret rat BG(m,n)} in the case where $m=1$] 
We prove the proposition in the case when $m=1$. Let $L$ be a Lie algebra over $\kappa$ which is a twisted form of $W(1;n)\otimes{\kappa}$. We have to show that there exists a Lie algebra $\widetilde{L}$ over $R$ such that $\widetilde{L}$ is a twisted form of $W(1;n)\otimes R$ satisfying $\widetilde{L}\otimes_R\kappa\simeq L$. By \cite[Theorem A]{Waterhouse}, $L$ is isomorphic to a Witt--Ree algebra of $D$-dimension $1$ with $\Dim_{\kappa}L=p^n$ on a purely inseparable $\kappa$-algebra $A=\kappa[x_1,\dots,x_n]/(x_1^p-a_1,\dots,x_n^p-a_n)$ of height one (cf.\,Definition \ref{def:Witt-Ree}). More precisely, there exists a derivation $D\in  \Der_{\kappa}(A)$ which satisfies the following conditions.
\begin{enumerate}
\renewcommand{\labelenumi}{(\roman{enumi})}
\item $fD=0$ with $f\in A$ implies that $f=0$.
\item $D(f)=0$ with $f\in A$ implies that $f\in\kappa$. 
\item the Lie subalgebra $L(A;D)\Def A\,D\subset\Der_{\kappa}(A)$ is a central simple Lie algebra over $\kappa$. 
\end{enumerate}
In addition, there exists an isomorphism $L\simeq L(A;D)$ of Lie algebras over $\kappa$. 
Therefore, we may assume that $L=L(A;D)$. First we fix any lifting $\widetilde{A}=R[x_1,\dots,x_n]/(x_1^p-\widetilde{a}_1,\dots,x_n^p-\widetilde{a}_n)$ of the purely inseparable algebra $A$ as in the case when $\bm{n}=\bm{1}$. Let $\pi\colon \widetilde{A}\twoheadrightarrow A=\widetilde{A}\otimes_R\kappa$ be the canonical surjective map. It suffices to show that there exists a lifting  $\widetilde{D}\in\Der_R(\widetilde{A})$ of the derivation $D\in\Der_{\kappa}(A)$ such that the Lie subalgebra $\widetilde{L}\Def L(\widetilde{A};\widetilde{D})\Def\widetilde{A}\, \widetilde{D}\subset\Der_R(\widetilde{A})$ is a twisted form of $W(1;n)\otimes R$.

Let $\frkM\subset A$ be the maximal ideal of $A$. Let $\widetilde{\frkM}\Def\pi^{-1}(\frkM)\subset\widetilde{A}$. Then $\widetilde{A}$ is a local ring with maximal ideal $\widetilde{\frkM}$. For each $1\le i\le n$, let $\delta_i\in\Der_R(\widetilde{A})$ be the derivation satisfying $\delta_i(x_j)=\delta_{ij}$. We have $\Der_{R}(\widetilde{A})=\sum_{i=1}^n\widetilde{A}\cdot\delta_i$ with free $\widetilde{A}$-basis $\{\delta_1,\dots,\delta_n\}$. Similarly, we have $\Der_{\kappa}(A)=\sum_{i=1}^nA\cdot\delta_i$ with free $A$-basis $\{\delta_1,\dots,\delta_n\}$. Let $D=\sum_{i=1}^mb_i\delta_i$ with $b_i\in A$. For each $i$, let $\widetilde{b}_i\in\widetilde{A}$ be any lift of $b_i$. By the condition (i) together with \cite[Lemma 3.4]{Ree}, we can find that $b_i\in A^*$ for some $1\le i\le n$. This implies that $\widetilde{b}_i\in \widetilde{A}\setminus\widetilde{\frkM}=\widetilde{A}^*$. Therefore, the derivation $\widetilde{D}\Def\sum_{i=1}^n\widetilde{b}_i\delta_i\in\Der_R(\widetilde{A})$ satisfies the condition (WR1) (cf.\ Section \ref{sec:Witt-Ree}). 

We can also see that the condition (WR2) in Definition \ref{def:Witt-Ree} is satisfied for $\widetilde{L}$. 
Indeed, by the condition (ii) for $D$, we have an exact sequence of $\kappa$-vector spaces
\[
0\lto\kappa\lto A\xrightarrow{~D~}D(A)\lto 0.
\]
This implies that $D(A)\subset A$ is a $\kappa$-subspace of $A$ with $\Dim_{\kappa}D(A)=p^n-1$. Hence, $A$ admits a basis $\{v_j\}_{j=1}^{p^n}$ such that 
\[
D(A)=\sum_{j=1}^{p^n-1}\kappa\cdot v_j.
\]
On the other hand, for any $f\in A$ and any lift $\widetilde{f}\in\widetilde{A}$ of $f$, we have $\pi(\widetilde{D}(\widetilde{f}))=D(f)$. This implies that $\pi(\widetilde{D}(\widetilde{A}))=D(A)$, i.e.\,$\widetilde{D}(\widetilde{A})\otimes\kappa=D(A)$. Thus we can take a minimal basis $\{\widetilde{v}_j\}_{j=1}^{p^n}$ of $\widetilde{A}$ with $\pi(\widetilde{v}_j)=v_j$ so that
\[
\widetilde{D}(\widetilde{A})=\sum_{j=1}^{p^n-1}R\cdot \widetilde{v}_j.
\]
As $\widetilde{A}$ is a free $R$-module of rank $p^n$, the minimal basis $\{\widetilde{v}_j\}_{j=1}^{p^n}$ is linearly independent over $R$. Therefore, $\widetilde{D}(\widetilde{A})$ is a free $R$-module of rank $p^n-1$. The exact sequence
\[
0\lto\Ker(\widetilde{D})\lto \widetilde{A}\xrightarrow{~\widetilde{D}~}\widetilde{D}(\widetilde{A})\lto 0
\]
then tells us that $\Ker(\widetilde{D})$ is a flat $R$-module, which implies that $\Ker(\widetilde{D})\otimes\kappa=\Ker(D)=\kappa$. By Nakayama's lemma, we can conclude that $\Ker(\widetilde{D})=R$ as desired.

Next we will discuss the condition (WR3) for $\widetilde{L}$. By condition (iii), $\widetilde{L}\otimes\kappa=L$ is central simple. Therefore, it suffices to show that $\widetilde{L}$ is center free. Indeed, for any $f,g\in A$ and any lifts $\widetilde{f},\widetilde{g}\in\widetilde{A}$ of $f,g$ respectively, we have
\[
ad(\widetilde{f}\widetilde{D})(\widetilde{g}\widetilde{D})=(\widetilde{f}\widetilde{D}(\widetilde{g})-\widetilde{g}\widetilde{D}(\widetilde{f}))\widetilde{D},
\]
hence $\pi(ad(\widetilde{f}\widetilde{D})(\widetilde{g}\widetilde{D}))=(fD(g)-gD(f))D=ad(fD)(gD)$. This implies that the reduction map $\End_{R\text{-lin}}(\widetilde{L})\twoheadrightarrow\End_{R\text{-lin}}(\widetilde{L})\otimes_R\kappa=\End_{\kappa\text{-lin}}(L)$ maps $ad(\widetilde{L})$ surjecitvely onto $ad(L)$. Therefore, by the same argument for the condition (WR2), we can show that $Z(\widetilde{L})=\Ker(ad\colon\widetilde{L}\to\End_{R\text{-lin}}(\widetilde{L}))$ is flat over $R$ with $Z(\widetilde{L})\otimes\kappa=0$. Again by Nakayama's lemma this implies that $Z(\widetilde{L})=0$ as desired.

We have seen that $\widetilde{L}$ is a Witt--Ree $R$-algebra on $\widetilde{A}$. It remains to prove that $\widetilde{L}$ is a twisted form of the generalized Witt algebra $W(1;n)\otimes_kR$. As $\uAut(W(1;n)\otimes_kR)=G(1;n)\otimes_kR$ is flat and of finite presentation over $R$, it suffices to show that $\widetilde{L}$ is a twisted form of $W(1;n)\otimes_kR$ in the fpqc topology (cf.\ \cite[Lemmas 02L0 and 02L2]{stack}). We identify the completion of the regular local ring $R$ with a ring of formal power series, i,e.\ $\widehat{R}=\kappa[[T_1,\dots,T_l]]$ and define $\overline{R}$ to be the integral closure of $\widehat{R}$ in an algebraic closure $\overline{\Frac\widehat{R}}$ of the fraction field $\Frac\widehat{R}$. As $\Spec\overline{R}\to\Spec R$ is a fpqc covering of $\Spec R$ (cf.\cite[03NV]{stack}), it suffices to show that $\widetilde{L}\otimes_R\overline{R}$ is isomorphic to $W(1;n)\otimes_k\overline{R}$. Let $\gamma_1,\gamma_p,\dots,\gamma_{p^{n-1}}\in\kappa$ be elements so that $d\gamma_1,d\gamma_p,\dots,d\gamma_{p^{n-1}}\in\Omega^1_{\kappa/k}$ are linearly independent over $\kappa$. We can take such elements thanks to the condition that $\Dim_{\kappa}\Omega^1_{\kappa/k}\ge n$. Let 
\[
B\Def{\overline{\kappa}[y_1,y_p,\dots,y_{p^{n-1}}]}/{(y_1^p-1,y_p^p-1,\dots,y_{p^{n-1}}^p-1)}
\]
and we define the derivation $\partial\in\Der_{\overline{\kappa}}(B)\subset\End_{\overline{\kappa}\text{-lin}}(B)$ by setting 
\[
\partial(y_{p^s})=\gamma_{p^s}y_{p^s}\quad (0\le s\le n-1). 
\]
Then, $L(B;\partial)\Def B\,\partial\subset\Der_{\overline{\kappa}}(B)$ is a Witt--Ree $\overline{\kappa}$-algebra on the purely inseparable algebra of height one $B$. Hence, by \cite[Theorems A and B]{Waterhouse}, there exists an isomorphism of $\overline{\kappa}$-algebras $\varphi\colon B\xrightarrow{\simeq}A\otimes\overline{\kappa}$ such that the conjugation $\varphi_*=\varphi\circ (-)\circ\varphi^{-1}$ by $\varphi$ gives an isomorphism of Lie algebras 
\[
\varphi_*\colon L(B;\partial)\xrightarrow{~\simeq~}L\otimes\overline{\kappa}.
\]
By construction, the characteristic polynomial $\varphi_E(t)=\varphi_{\partial}(t)\in\overline{\kappa}[t]$ of the derivation $E\Def\varphi_*(\partial)=\varphi\circ\partial\circ\varphi^{-1}\in L\otimes\overline{\kappa}\subset\Der_{\overline{\kappa}}(A\otimes\overline{\kappa})\subset\End_{\overline{\kappa}\text{-lin}}(A\otimes\overline{\kappa})$ is a separable polynomial of degree $p^n$. 
The roots of $\varphi_{E}(t)$ form an $n$-dimensional $\F_p$-vector subspace $\sum_{s=0}^{n-1}\F_p\gamma_{p^s}$ in $\overline{\kappa}$.

Let $\widetilde{E}\in\widetilde{L}\otimes_R\overline{R}$ be any lift of $E$. Then the characteristic polynomial $\varphi_{\widetilde{E}}(t)\in \overline{R}[t]$ is a separable polynomial of degree $p^n$. 
For each $0\le s\le n-1$, we denote by $\lambda_{p^s}\in\overline{R}$ the root of $\varphi_{\widetilde{E}}(t)$ whose image in $\overline{\kappa}$ coincides with $\gamma_{p^s}$. Moreover, for each $0\le i<p^n$, if $i=i_0+i_1p+\cdots+i_{p^{n-1}}p^{n-1}$ is the $p$-adic expansion of $i$, we set 
\[
\lambda_i\Def i_0\lambda_1+i_1\lambda_p+\cdots+i_{p^{n-1}}\lambda_{p^{n-1}}
\]
and let $W_i\subseteq \widetilde{A}\otimes_R\overline{R}$ be the eigenspace with respect to the characteristic root $\lambda_i$. Then each $W_i$ is a free module of rank $1$ and we have the decomposition
\[
\widetilde{A}\otimes_R\overline{R}=\sum_{i=0}^{p^n-1}W_i.
\]
As the decomposition is compatible with the decomposition of $A\otimes\overline{\kappa}$ into the eigenspaces with respect to the operator $E$, by \cite[Lemma 6.1]{Ree}, for each $i$, there exists an element $u_i\in W_i\cap(\widetilde{A}\otimes_R\overline{R})^*$ with $u_0=1$.  By condition (WR2) for $\widetilde{L}$, we have $u_i^p\in \overline{R}^*$ for any $i$. Therefore, by multiplying invertible elements of $\overline{R}$, we may assume that $u_i^p=1$ for any $i$. Note also that $\lambda_i\in \overline{R}^*$ for $i\neq 0$. As $\gamma_1,\gamma_p,\dots,\gamma_{p^{n-1}}\in\kappa$ are algebraically independent, $\overline{R}$ contains the fraction field $K_0\Def k(\lambda_1,\lambda_p,\dots,\lambda_{p^{n-1}})$ of the subring $k[\lambda_1,\lambda_p,\dots,\lambda_{p^{n-1}}]$ (cf.\ Proof of \cite[Lemma 28.3(iii)]{Matsumura}), hence contains its algebraic closure $\overline{K}_0$ as well. 
If we define the $K_0$-subalgebra $\widetilde{B}\subset\widetilde{A}\otimes_R\overline{R}$ to be 
\[
\widetilde{B}\Def \overline{K}_0[u_1,u_p,\dots,u_{p^{n-1}}],
\]
then $\widetilde{E}\in\Der_{\overline{K}_0}(\widetilde{B})$ and we have $\widetilde{L}\otimes_R\overline{R}=L(\widetilde{B};\widetilde{E})\otimes_{\overline{K}_0}\overline{R}$. However, by \cite[Theorem 6.10]{Ree}, the Witt--Ree $\overline{K}_0$-algebra $L(\widetilde{B};\widetilde{E})=\widetilde{B}\,\widetilde{E}$ is isomorphic to  $W(1;n)\otimes_k \overline{K}_0$. This shows that $\widetilde{L}\otimes_R\overline{R}$ is isomorphic to $W(1;n)\otimes_k\overline{R}$, which completes the proof of the proposition. 
\end{proof}

\begin{DA}
Not Applicable. 
\end{DA}

\begin{COI}
The author has no conflict of interest to declare that are relevant to this article.
\end{COI}


\end{document}